\documentclass{sn-jnl}
\usepackage[utf8]{inputenc}

\usepackage{manyfoot}
\usepackage{xcolor}
\usepackage{amsthm}

\usepackage{amssymb}
\usepackage{manyfoot}
\usepackage{xcolor}
\usepackage{url}
\usepackage{amsmath}
\usepackage{amsthm}
\usepackage{proof}
\usepackage[all]{xy}

\usepackage{enumitem}

\newcommand{\forces}{\Vdash}
\newcommand{\Set}{\mbox{\bf Set}}

\newcommand{\worlds}{\mathcal W}

\newcommand{\flatOp}[1]{{#1}^\flat}
\newcommand{\natOp}[1]{{#1}^\natural}

\newcommand{\conssym}{\triangleright}
\newcommand{\cons}{\ifmmode\mathrel{\conssym}\else\mbox{$\conssym$}\fi}
\newcommand{\valid}{\blacktriangleright}
\newcommand{\sem}[1]{\mbox{$[\hspace{-0.05cm}[$}#1\mbox{$]\hspace{-0.05cm}]$}}
\newcommand{\semwp}[1]{\sem{#1}^{\mathcal{W}'}}
\newcommand{\vsem}[1]{\mbox{$\{\hspace{-0.05cm}\mid$} #1\mbox{$\mid\hspace{-0.05cm}\}$}}

\newcommand{\produ}[1]{\langle #1 \rangle}
\newcommand{\col}{\!:\!}

\newcommand{\kripkeForces}{\forces}
\newcommand{\kripkeWorlds}{{\mathcal K}}
\newcommand{\op}{^{\mbox{$op$}}}
\newcommand{\up}{^{\uparrow}}

\newtheorem{theorem}{Theorem}                         
\newtheorem{definition}[theorem]{Definition}                                                     
\newtheorem{lemma}[theorem]{Lemma}                             
                 
\newtheorem{proposition}[theorem]{Proposition}

\newcommand{\base}[1]{\ensuremath{\mathcal{#1}}}
\newcommand{\calculus}[1]{\ensuremath{\mathsf{#1}}}

\newcommand{\ZFset}[1]{\ensuremath{\mathbb{#1}}}
\newcommand{\setAtoms}{\ZFset{A}}

\newcommand{\sat}{\Vdash}

\newcommand{\supp}{\Vdash}

\newcommand{\at}[1]{{ #1}}

\begin{document}

\title{Categorical Proof-theoretic Semantics}




\author*[1,2]{\fnm{David} \sur{Pym}}\email{d.pym@ucl.ac.uk}
\equalcont{These authors contributed equally to this work.}

\author*[3]{\fnm{Eike} \sur{Ritter}}\email{e.ritter@bham.ac.uk}
\equalcont{These authors contributed equally to this work.}

\author*[4]{\fnm{Edmund} \sur{Robinson}}\email{e.p.robinson@qmul.ac.uk}
\equalcont{These authors contributed equally to this work.}

\affil*[1]{\orgdiv{Computer Science and Philosophy}, \orgname{University College London}, \orgaddress{\street{Gower Street}, 
\city{London}, \postcode{WC1E 6BT}, 
\country{UK}}}

\affil[2]{\orgdiv{Institute of Philosophy}, \orgname{School of Advanced Study, University of London}, \orgaddress{\street{Malet Street}, 
\city{London}, \postcode{WC1E 7HU}, 
\country{UK}}}

\affil[3]{\orgdiv{Department of Computer Science}, \orgname{University of Birmingham}, \orgaddress{\street{Edgbaston}, 
\city{Birmingham}, \postcode{B15 2TT}, 
\country{UK}}}

\affil[3]{\orgdiv{School of Electronic Engineering and Computer Science}, \orgname{Queen Mary, University of London}, \orgaddress{\street{Mile End Road}, 
\city{London}, \postcode{E1 4NS}, 
\country{UK}}}

\abstract{
In proof-theoretic semantics, 
model-theoretic validity is replaced by proof-theoretic validity. Validity 
of formulae is defined inductively from a base giving the validity of atoms 
using inductive clauses derived from proof-theoretic rules. A key aim is 
to show completeness of the proof rules without any requirement for formal 
models. Establishing this for propositional intuitionistic logic (IPL) 
raises some technical and conceptual issues.  
We relate Sandqvist's (complete) base-extension semantics of intuitionistic 
propositional logic 
to categorical proof theory in presheaves, reconstructing categorically the 
soundness and completeness arguments, thereby demonstrating the naturality of 
Sandqvist's constructions. This naturality includes Sandqvist's treatment of disjunction
that is based on its second-order or elimination-rule presentation. 
These constructions embody not just validity, but certain forms of objects of 
justifications. This analysis is taken a step further by showing that from the 
perspective of validity, Sandqvist's semantics can also be viewed as the natural 
disjunction in a category of sheaves. 
}

\maketitle

\section{Introduction} \label{sec:intro} 

In model-theoretic semantics, logical languages are interpreted in 
mathematical structures that carry appropriate axiomatizations and 
support a notion of model-theoretic consequence. Specifically, a 
propositional sentence $\phi$ is a model-theoretic consequence 
of a set $\Gamma$ of propositional sentences iff every model $\mathcal{M}$ of 
$\Gamma$ is also a model of $\phi$; that is, 
\[
\begin{array}{rcl}
    \Gamma \models \phi & \mbox{\rm iff} & 
        \mbox{\rm for all $\mathcal{M}$, if, for all $\psi \in \Gamma$, if 
        $\mathcal{M} \models \psi$, then $\mathcal{M} \models \phi$ } 
\end{array}
\]
As Schroeder-Heister \cite{Schroeder2007modelvsproof} explains, in this definition of consequence 
we have a transmission of truth from premisses to conclusion in which 
transmission is determined by classical implication in the meta-theory. 
This definition of consequence supports inductive characterizations of 
the meanings of the logical constants. For example, 
\[
 \begin{array}{rcl}
    \mathcal{M} \models \phi \wedge \psi & \mbox{\rm iff} & 
        \mbox{\rm $\mathcal{M} \models \phi$ and $\mathcal{M} \models \psi$}   
 \end{array}
\]

Proof-theoretic consequence is usually defined as derivability in a 
formal system: a propositional sentence $\phi$ is a proof-theoretic 
consequence of a set of propositional sentences $\Gamma$ in a formal 
system $S$ --- that is, $\Gamma \vdash_S \phi$ ---  if it can be be 
derived from the elements of $\Gamma$ using the axioms and inference 
rules of $S$. This definition of consequence also supports inductive 
characterizations of the meanings of the logical constants. For example, 
\[
 \begin{array}{rcl}
    \mbox{\rm if $\Gamma \vdash_S \phi$ and $\Gamma \vdash_S \psi$, then 
        $\Gamma \vdash_S \phi \wedge \psi$}   
 \end{array}
\]
Note that the invertibility of these characterizations is a delicate 
matter (see, for example, \cite{NegriVonPlato}).

From the perspective of model-theoretic semantics, and its primary 
notion of truth, the correctness of inferences in a formal system $S$ 
is given by a soundness assertion, 
\[
  \begin{array}{rcl}
    \mbox{\rm if $\Gamma \vdash_S \phi$, then $\Gamma \models \phi$}
  \end{array}
\]
If completeness --- the converse --- also holds, then model-theoretic 
consequence and proof-theoretic consequence coincide. Note that 
completeness does not imply the invertibility of the proof-theoretic 
characterizations. 

Proof-theoretic semantics provides an alternative account of the 
meaning of the logical constants to that which is provided by 
model-theoretic semantics. Roughly speaking, there are two approaches 
to proof-theoretic semantics. First, there is what we call 
Dummett-Prawitz proof-theoretic semantics \cite{Schroeder2007modelvsproof} 
and, second, there is what we call base-extension semantics \cite{Sandqvist2015IL,Piecha2016completeness,Piecha2015failure}. The former has 
temporal priority, but the latter may be seen as being the more 
general perspective \cite{Alex:PtV-BeS}. 

The Dummett-Prawitz view arises from a philosophical reading by Dummett 
\cite{Dummett1991logical} of the normalization results by Prawitz 
\cite{Prawitz1965}. It has subsequently been developed substantially both 
mathematically and philosophically with Schroeder-Heister 
\cite{schroeder2006validity} giving a general account that clearly 
separates the semantic and computational considerations. It has largely 
been developed for IL. In this context, the rules of the natural 
deduction system $\calculus{NJ}$ for IL are taken to be \emph{a priori} 
valid. An argument is a tree of formulas whose leaves are called its 
assumptions, some of which may be labelled as discharged. An argument 
is \emph{indirect} (i.e., not direct) if it contains \emph{detours} 
in which case it may be made \emph{direct} by reduction 
\emph{\`a la } Prawitz \cite{Prawitz1965}; for example,  
the reduction of an indirect argument via $\to$ may be reduced as follows:
\[
\infer{\psi}{
   \deduce{\phi}{\mathcal{D}_1} & 
   \infer{\phi \supset \psi}
       {
           \deduce{\psi}{\deduce{\mathcal{D}_2}{[\phi]}}
       }
   }
   \quad \rightsquigarrow \quad
   \deduce{\psi}{
        \deduce{\mathcal{D}_2}{
           \deduce{\phi}{\mathcal{D}_1}
           }
       }
\]
Arguments without assumptions and without detours are said to be 
\emph{canonical proofs}, they are inherently valid. The validity of 
an arbitrary argument is determined by whether or not it represents 
according to some fixed operations (e.g., through reduction) one of 
these canonical proofs. In this way, proof-theoretic validity in 
the Dummett-Prawitz tradition is a semantics of proofs. This view 
is not our concern here. 

Base-extension semantics is a characterization of consequence given 
by an inductively defined judgment whose base case is given by proof 
rather than by truth. Crucially, while in model-theoretic semantics 
the base case of the judgment is given by truth (i.e., in the model 
$\mathcal{M}$ with interpretation $I$, the judgment $w \sat \at{p}$ 
obtains iff $w \in I(\at{p})$),  in proof-theoretic semantics 
it is instead given by provability in an atomic system; that is, 
by the following clause in which $\base{B}$ is an arbitrary atomic 
system:
 \[
    \supp_\base{B} p \quad \mbox{iff} \quad \mbox{$p$ is provable in \base{B}} 
 \]
Note that base-extension semantics relies on a (proof-irrelevant) provability judgement not on the judgement that a proof-object 
establishes a consequence. 
 
 Of course, one can give base-extension semantics that simply mimic 
 model-theoretic semantics. Goldfarb's completeness proof works by taking 
 a base-extension semantics that encodes the possible worlds structure of 
 a Kripke countermodel --- see Goldfarb \cite{goldfarbdummett2016}. This 
 approach allows completeness to be obtained while still using the usual frame semantics 
 for disjunction in IPL as given by Kripke \cite{kripke1965semantical}, 
 \[
    w \supp \phi \lor \psi \quad \text{ iff } \quad w \supp \phi \text{ or } w \supp \psi , 
 \] 
where $w$ denotes an arbitrary world. Similarly, recent work by Stafford and Nascimento \cite{SN23} obtains completeness by constructing a Kripke model from sets of base rules. 

 By contrast, in the complete base-extension semantics for IPL given by Sandqvist 
 \cite{Sandqvist2015IL}, disjunction has the following clause in which $\setAtoms$ 
 is the set of atomic propositions:
 \[
 \begin{array}{r@{\quad}c@{\quad}l}
  \sat_{\base{B}} \phi \lor \psi & \mbox{\rm iff} & 
    \mbox{for any $\base{C} \supseteq \base{B}$ and any $\at{p} \in \setAtoms$,}\\ 
    & & \mbox{\rm $\phi \supp_\base{C} \at{p}$ and $\psi \supp_\base{C} \at{p}$ implies  $\emptyset \supp_\base{C} \at{p}$} 
 \end{array}
 \]
 Completeness is obtained by using a base that mimics propositions and their proofs, rather than the elements of a Kripke model. 

While, in the spirit of Dummett-Prawitz proof-theoretic semantics, the form of the semantics for 
disjunction given above is naturally seen as a representation of the $\vee$-elimination rule of natural deduction \cite{Gentzen1935,Prawitz1965}, it is also naturally seen as the well-known definition of $\vee$ in second-order propositional logic see, for example,  \cite{GLT89,TS2000,F2006,Sandqvist2008,Sandqvist2015IL}). 
This view is also discussed in Section~\ref{sec:kripke}.

In both cases, the definition restricts the conclusion of the hypothetical assumptions and the conclusion of the definition itself to be atomic. That is,
\[
    \frac{\begin{array}{c@{\quad}c@{\quad}c}
                        & [\phi] & [\psi] \\ 
                        & \vdots & \vdots \\ 
         \phi \vee \psi & p      & p  
          \end{array}}{\hspace{11mm} p}
\]
and, writing $\bigwedge$ for the second-order propositional universal quantifier, 
\[
    \phi \vee \psi \quad = \quad 
    \bigwedge p \,.\, (\phi \supset p) \supset (\psi \supset p) \supset p 
\]
Both Sandqvist's semantics and our category-theoretic analysis of it can be understood in both these ways. 

This way of interpreting disjunction has been studied 
in the context of the atomic fragment $F_{at}$ 
\cite{FF2013a,FF2013b,FF2015,FES2020,PTP2022}  
of System F \cite{Girard71,Girard72}. In \cite{FF2013b,FF2015} it is  shown that this interpretation is faithful; that is, for any formula $\phi$ of IPL, $\phi$ is provable iff its $F_{at}$-interpretation is provable in $F_{at}$. This result might, from a perspective that lies outwith our present scope, be seen as a proof-theoretic counterpart to Sandqvist's completeness result. The link between $F_{at}$ and our account is
explained in Section~\ref{sec:kripke}. Furthermore, this form of 
the semantics for disjunction is closely related Beth's semantics; 
see, for example, the discussion of Kripke-Beth-Joyal semantics 
and the relevant historical context in \cite{LS86}. 



Gheorghiu and Pym \cite{Alex:PtV-BeS} have shown that base-extension semantics may be regarded 
as the declarative counterpart to the operational paradigm of proof-theoretic validity. 
In particular, it is shown in \cite{Alex:PtV-BeS} that Dummett-Prawitz validity can be recovered 
from base-extension semantics for intuitionistic propositional logic. 
 
Category theory is a general theory of mathematical structures and their relations 
that was introduced by Eilenberg and Mac~Lane in the middle of the 
20th century in their foundational work on algebraic topology. It provides a unifying 
language for studying mathematical structures that supports the use of 
one theory to explain another. Categorical logic is a highly developed area that 
studies logic from the point of view of categorical structures. 

The relationship between model-theoretic semantics and categorical logic and between 
proof theory and categorical logic is rich and highly developed, 
especially in the world of intuitionistic and modal logics. See, for example, 
\cite{MM1992,BlackburnEtAl,Jacobs,Seely83}. The relationship between 
the proof theory of classical logic is less well developed, but see, for 
example, \cite{Bellin2006CategoricalPT,Robinson,FuhrmannPym}. 

In this paper, we begin an exploration of the relationship between 
proof-theoretic semantics and categorical logic. There are two main 
motivations for this. First, in general, to bring proof-theoretic 
semantics into the framework of categorical logic. Second, more 
specifically, to explore certain technical aspects of the formulation 
of the base-extension semantics of intuitionistic propositional logic 
\cite{Sandqvist2015IL,Piecha2016completeness,Piecha2015failure,goldfarbdummett2016,SN23}: 
\begin{itemize}
\item[--] the formal naturality of the semantics
\item[--] the choices of semantics for disjunction. 
\end{itemize}
Here, formal naturality refers to the existence of natural transformations
between functors (as defined in \cite{LS86,MM1992,Jacobs}, for example): 
if $F$ and $G$ are functors between categories $C$ and $D$, then a natural 
transformation $\eta$ between $F$ and $G$ is family of morphisms that 
satisfies the following: 
\begin{enumerate}
    \item $\eta$ must associate to every object $x$ in $C$ an arrow 
    $\eta_x : F(x) \rightarrow G(x)$ 
    \item for every $f:x \rightarrow y$ in $C$, $\eta_y \circ F(f) = G(f) \circ \eta_x$, 
    where $\circ$ denotes composition of morphisms. 
\end{enumerate}
Informally, the notion of a natural transformation captures that a 
given map between functors can be done consistently over an entire category. 
In the situation above, we refer to the structure being `natural in $x$'. 

Our primary focus will be Sandqvist's base-extension semantics 
for intuitionistic propositional logic \cite{Sandqvist2015IL}. We give a concise 
introduction to this work in Section~\ref{sec:p-ts-IPL}. 








We now summarize the structure of the remainder of this paper. In Section~\ref{sec:p-ts-IPL}, we summarize Sandqvist's base-extension semantics for intuitionistic propositional logic, including the 
soundness and completeness of consequences derivable in NJ (recalled in  Figure~\ref{fig:nj}), denoted $\Gamma \vdash \phi$, for base-extension 
validity, denoted $\Gamma \Vdash \phi$.  We proceed, in Section~\ref{sec:categorical}, to give a categorical formulation of Sandqvist's constructions. This uses well-established (essentially type-theoretic) ideas from categorical logic (e.g., \cite{Seely83,LS86,MM1992,Jacobs}), working in a presheaf category. We establish the correctness (soundness and completeness) of our algebraic constructions relative to Sandqvist's semantics. Next, in Section~\ref{sec:kripke}, 
%
%
%
we discuss the relationships between base-extension semantics, Kripke models of intuitionistic logic, and presheaf models of type theory. In Section~\ref{sec:metatheory}, we reconstruct categorically the logical metatheory of the base-extension semantics for intuitionistic propositional logic, and establish the familiar soundness and completeness theorems with respect to NJ: that is, that $\Gamma \vdash \phi$ iff $\Gamma \models \phi$, where $\models$ denotes the consequence relation derived from our category-theoretic model. The paper concludes with a reflection, in Section~\ref{sec:disjunction}, on the semantics of disjunction, explaining how the choice of its interpretation --- essentially between the Kripke-style interpretation and the Sandqvist-style interpretation, which adopts the second-order formulation that corresponds to the proof-theoretic interpretation offered by the elimination rule --- affects completeness. 

A preliminary version of this work has been presented at the 11th Scandinavian Logic Society Symposium, 2022 \cite{Pym2022catpts}. 

\section{Base-extension Semantics for Intuitionistic Propositional Logic} \label{sec:p-ts-IPL}


Sandqvist \cite{Sandqvist2015IL} gives a base-extension proof-theoretic semantics 
for IPL for which natural deduction is sound and complete. We refer the reader 
to \cite{Sandqvist2015IL} for the detailed motivation and technical development. 
Here we give a very brief summary.

In the sequel, Romans $p$, $P$, etc., respectively denote atoms and sets of atoms; 
Greeks $\phi$, $\Gamma$, etc., respectively denote formulae and sets of formulae.  

A base $\mathcal{B}$ is a set of atomic rules (for $\vdash_\mathcal{B}$), as in Definition~\ref{def:base}, 
which also defines the application of base rules, and satisfaction in a base ($\Vdash_\mathcal{B}$). 

\begin{definition}[Base] \label{def:base} Base rules $\mathcal{R}$, application of base rules, 
and satisfaction of formulae in a 
(possibly finite) countable base $\mathcal{B}$ of rules $\mathcal{R}$ are defined as follows:
an atomic rule $\mathcal{R}$ is given as a second-order implication (denoted $\Rightarrow$) involving (sets of) atoms: $((P_1 \Rightarrow q_1) , \ldots , (P_n \Rightarrow q_n)) \Rightarrow r)$ (upper case $P_i$ denote sets of atoms, while lower case $q_j$ denote individual atoms). 
This can be thought of as denoting the proof fragment with discharged hypotheses: 
\[
\frac{\begin{array}{ccc}
        [P_1] &        & [P_n] \\
        q_1   & \ldots & q_n  
      \end{array}
}{r} \, \mathcal{R}
\]
We can now define consequence for atoms with respect to a base: 
\[
\begin{array}{rl}
\mbox{\rm (Ref)} & \mbox{\rm $P , p \vdash_\mathcal{B} p$} \\ 
(\mbox{\rm App}_\mathcal{R}) & \mbox{\rm if $((P_1 \Rightarrow q_1) , \ldots , (P_n \Rightarrow q_n)) \Rightarrow r)$ is a rule in the base and,} \\
    & \mbox{\rm for all $i \in [1,n]$, $P , P_i \vdash_\mathcal{B} q_i$, then $P \vdash_\mathcal{B} r$} 
\end{array}
\]
\end{definition}
Here, the combinators $\mbox{\rm (Ref}$ and $\mbox{\rm App}_\mathcal{R})$ specify how to 
construct proofs in a base using rules $\mathcal{R}$ expressed using the 
second-order implication $\Rightarrow$ as 
discussed above (see also Sandqvist \cite{Sandqvist2015IL}).  

\begin{definition}[Validity in a Base] \label{def:validity-in-a-base}
Validity in a base is defined inductively as follows: 
\[
\begin{array}{rl} 
\mbox{\rm (At)} &\mbox{\rm For atomic $p$, $\Vdash_\mathcal{B}p$ iff $\vdash_\mathcal{B} p$} \\ 
(\supset) & \mbox{\rm $\Vdash_\mathcal{B} \phi \supset \psi$ iff $\phi \Vdash_\mathcal{B} \psi$} \\ 
(\wedge) & \mbox{\rm $\Vdash_\mathcal{B} \phi \wedge \psi$ iff $\Vdash_\mathcal{B} \phi$ and $\Vdash_\mathcal{B} \psi$} \\     
 (\vee) & \mbox{\rm $\Vdash_\mathcal{B} \phi \vee \psi$ iff, for every atomic $p$ and every $\mathcal{C} \supseteq  \mathcal{B}$, } 
         \mbox{\rm if $\phi \Vdash_\mathcal{C} p$ and $\psi \Vdash_\mathcal{C} p$,} \\
         & \mbox{then $\Vdash_\mathcal{C} p$}\\
(\bot) & \mbox{\rm $\Vdash_\mathcal{B} \bot$ iff, for all atomic $p$, $\Vdash_\mathcal{B} p$} \\
\mbox{\rm and} & \\
\mbox{\rm (Inf)} & \mbox{\rm for $\Theta \neq \emptyset$, 
        $\Theta \Vdash_\mathcal{B} \phi$ iff, for every $\mathcal{C} \supseteq \mathcal{B}$, 
        if $\Vdash_\mathcal{C} \theta$ for every $\theta \in \Theta$,} \\ 
        & \mbox{\rm  then $\Vdash_\mathcal{B} \phi$} 
\end{array}
\]
\end{definition} 

\begin{definition}[Validity] \label{def:validity}
Define (cf.~\cite{Sandqvist2015IL}) $\Gamma \Vdash \phi$ as: for all $\mathcal{B}$, 
if $\Vdash_\mathcal{B} \psi$ for all $\psi \in \Gamma$, then $\Vdash_\mathcal{B} \phi$.
\end{definition}

Write $\Gamma \vdash \phi$ to denote that $\phi$ is provable from $\Gamma$ in the natural deduction 
systems NJ (cf.~\cite{Sandqvist2015IL}). For reference, NJ is summarized in Figure~\ref{fig:nj}. 

\begin{figure}[t]
\hrule 
\[
\begin{array}{c@{\quad\quad}c}
    & \infer[\bot{E}]{\Gamma \vdash \phi}{\Gamma \vdash \bot} \\ [4pt] 
\infer[\wedge{I}]{\Gamma \vdash \phi \wedge \psi}{\Gamma \vdash \phi & \Gamma \vdash \psi}
    &  \infer[\wedge{E}]{\Gamma \vdash \phi}{\Gamma \vdash \phi \wedge \psi} \quad 
            \infer[\wedge{E}]{\Gamma \vdash \psi}{\Gamma \vdash \phi \wedge \psi} \\ [4pt] 
\infer[\vee{I}]{\Gamma \vdash \phi \vee \psi}{\Gamma \vdash \phi} \quad 
            \infer[\vee{I}]{\Gamma \vdash \phi \vee \psi}{\Gamma \vdash \psi}     
    &  \infer[\vee{E}]{\Gamma \vdash \chi}
            {\Gamma \vdash \phi \vee \psi & \Gamma, \phi \vdash \chi &
              \Gamma, \psi \vdash \chi} \\ [4pt]
\infer[\supset{I}]{\Gamma \vdash \phi \supset \psi}{\Gamma , \phi \vdash \psi}
    & \infer[\supset{E}]{\Gamma \vdash \psi}{\Gamma \vdash \phi & \Gamma \vdash \phi \supset \psi} \\ [4pt] 
\end{array}
\]
\hrule 
\caption{The calculus NJ (in sequential form and eliding $\top$) \cite{Prawitz1965}}
\label{fig:nj}
\end{figure}

\begin{theorem}[Soundness] \label{thm:BaseSoundness}
If $\Gamma \vdash \phi$, then $\Gamma \Vdash \phi$. 
\end{theorem} 

\begin{theorem}[Completeness] \label{thm:BaseCompleteness}
If $\Gamma \Vdash \phi$, then $\Gamma \vdash \phi$. 
\end{theorem} 

Sandqvist's completeness theorem \cite{Sandqvist2015IL} makes essential use of
\emph{flattening}. 

\begin{definition}[The operations $\flatOp{-}$ and $\natOp{-}$] \label{def:flatnat}
Let $\Delta$ contain all elements of $\Gamma \cup 
\{\phi\}$ and their subformulae. With every non-atomic $\delta \in
\Delta$, associate a distinct atomic $\delta^\flat \not \in \Delta$
and, for every atomic $q \in \Delta$, take $q^\flat = q$. We write
$\Delta^\flat$ for the flattening of $\Delta$. We also require the
inverse operation. For any atom $p$, define $\natOp{p}$ to be $\phi$
if $\flatOp{\phi} = p$ and $\natOp{p} = p$, otherwise.
\end{definition}

Sandqvist \cite{Sandqvist2015IL} defines a special base $\mathcal{N}$ depending on $\Delta^\flat$ as follows: 

\begin{definition}[The base $\mathcal{N}_{\Delta^\flat}$] \label{def:N}
$\mathcal{N}_{\Delta^\flat}$ is defined as the base containing exactly the following rules, corresponding to NJ (Figure~\ref{fig:nj}):
\begin{itemize}
\item[$\supset$I:] $(\phi^\flat \Rightarrow \psi^\flat) \Rightarrow (\phi \supset \psi)^\flat$  
\item[$\supset$E:] $(\phi \supset \psi)^\flat , (\Rightarrow \phi^\flat) \Rightarrow \psi^\psi$
\item[$\wedge$I:] $(\Rightarrow \phi^\flat) , (\Rightarrow \psi^\flat) \Rightarrow (\phi \wedge \psi)^\flat$ 
\item[$\wedge$E:] $(\Rightarrow (\phi \wedge \psi)^\flat) \Rightarrow \phi^\flat$ 
\item[$\wedge$E:] $(\Rightarrow (\phi \wedge \psi)^\flat) \Rightarrow \psi^\flat$ 
\item[$\vee$I:] $(\Rightarrow \phi^\flat) \Rightarrow (\phi \vee \psi)^\flat$
\item[$\vee$I:] $(\Rightarrow \psi^\flat) \Rightarrow (\phi \vee \psi)^\flat$
\item[$\vee$E:] $(\Rightarrow (\phi \vee \psi)^\flat) , (\phi^\flat \Rightarrow p) , (\psi^\flat \Rightarrow p) \Rightarrow p$
\item[$\bot$:] $(\Rightarrow \bot^\flat) \Rightarrow p$
\end{itemize}
\end{definition}

Before embarking on our category-theoretic 
analysis, we remark that in Appendix~\ref{sec:consequence-relations} we discuss the relationship between the base $\mathcal{N}_{\_}$ and how it is that validity is represented by consequence 
relations. 

\section{A Categorical Interpretation} \label{sec:categorical} 

There is a well-established way of interpreting natural deduction proofs in 
intuitionistic logic (NJ) in categories that have structure corresponding to the 
logical connectives \cite{Seely83,LS86,Jacobs}. Specifically, we interpret proofs in NJ in bicartesian closed 
categories, in which products are used to interpret conjunction, exponentials 
(function spaces) are used to interpret implication, and coproducts 
are used to interpret disjunction. Thus we obtain an interpretation of 
the following form: a morphism 
\[
    \sem{\Gamma} \stackrel{\sem{\Phi}}{\longrightarrow} \sem{\phi}
\]
interprets a proof $\Phi$ of the consequence $\Gamma \vdash \phi$. Here, $\Phi$ 
is a proof in NJ of the consequence, represented using the terms of a language 
such as the typed $\lambda$-calculus with products and sums. 

In this set-up, we work with judgements of the form 
\[
    x_1:\phi_1 , \ldots , x_i:\phi_i , \ldots , x_m:\phi_m \vdash \Phi(x_1, \ldots , x_m) : \phi   
\]
which are read, in the sense of the propositions-as-types interpretation, 
as follows: if the $x_i$s are witnesses for proofs of the $\phi_i$s, 
then $\Phi(x_1,\ldots,x_m)$ denotes a proof of $\phi$ constructed using 
the rules of NJ. 

Now, if $\Phi_i$ is a specific proof of $\phi_i$, then it can be substituted 
for $x_i$ throughout this judgement to give
\[
    x_1:\phi_1 , \ldots , x_m:\phi_m \vdash \Phi(x_1 , \ldots , x_m)[\Phi_i/x_i] : \phi   
\]
where the assumption $x_i:\phi_i$ has been removed and the occurrence of 
$x_i$ in $\Phi$ has been replaced by $\Phi_i$. 

In the setting of proof-theoretic semantics, we are concerned in the 
first instance with derivations that are restricted to the rules of a base. 
We introduce terms for derivations --- which can be seen as a restricted 
class of the terms described above --- in a base as follows:
\[
  \Phi ::= x \mid \Phi_\mathcal{R}(\Phi_1, \ldots, \Phi_m)
\]
where, as in Definition~\ref{def:base}, $\mathcal{R}$ is a rule
\[
\frac{\begin{array}{ccc}
        [P_1] &        & [P_n] \\
        q_1   & \ldots & q_n  
      \end{array}
}{r} \, \mathcal{R}
\]
of a base $\mathcal{B}$, $x$ is a witness 
for a derivation, and the $\Phi_i$s denote, inductively, derivations of atoms 
$p$, as given in Definition~\ref{def:base}. 

If $P = p_1, \ldots, p_m$ and $X = x_1, \ldots, x_m$, we write $(X \col P)$ for 
$x_1 \col p_1, \ldots, x_m \col p_m$. Using this notation, well-formed derivations 
in a base are given inductively by
\begin{gather*}
  \infer[\mbox{\rm (Ref)}]{(X \col P), x \col p\vdash_{\mathcal{B}} x \col p}{} \\[0.5cm]
  \infer[\mbox{\rm (App}_\mathcal{R})]{(X \col P) \vdash_{\mathcal{B}} \Phi_\mathcal{R} (\Phi_1, \ldots, \Phi_n)
    \col  r}
  {(X \col P),(X_i \col  P_i) \vdash_{\mathcal{B}}\Phi_i \col  q_i \quad i = 1 , \ldots , n
  }
\end{gather*}
Note that in the rule $\mbox{\rm App}_\mathcal{R}$ the variables $X_i$ are bound in the right-hand side of the conclusion. 

Writing $\Psi[\Phi/x]$ for the term obtained by substituting $\Phi$
for $x$ in $\Psi$, we have two key substitution lemmas. 
First, substitution preserves consequence:  
\begin{lemma}
  If $(X \col P)\vdash_{\mathcal{B}} \Phi \col  q$ and 
  $(X \col P), y \col  q \vdash_{\mathcal{B}} \Psi \col  r$, 
  then $(X \col  P) \vdash_{\mathcal{B}} \Psi[\Phi/y] \col  r$.
\end{lemma}
\begin{proof}
  By induction over the derivation of $\Psi$.
  \begin{description}[leftmargin=5mm,nosep]
    \item[$\Psi = y $] Suppose $(X \col P), y \col  q \vdash_{\mathcal{B}}
     y\col q$. By definition, $\Psi[\Phi/y] = \Phi$. Hence, by assumption, 
      $(X\col  P) \vdash \Psi[\Phi/y] \col r$.
    \item[$\Psi = \Phi_\mathcal{R}(\Phi_1, \ldots, \Phi_n)$] By the induction
      hypothesis, $(X \col P),(X_i \col  P_i)
      \vdash_{\mathcal{B}}\Phi_i[\Phi/y]\col  q_i$. Hence also
      $(X \col  P) \vdash_{\mathcal{B}} \Phi_\mathcal{R}(\Phi_1[\Phi/y], \ldots,
      \Phi_m[\Phi/y]) \col  r$.
  \end{description}
\end{proof}

Second, substitution is associative:  
\begin{lemma}
\label{lemmaAssoc}
For all derivations $(X \col  P) \vdash_{\mathcal{B}}\Phi_1 \col 
q_1$, $(X \col   P), y_1 \col   q_1 \vdash_{\mathcal{B}}\Phi_2
\col  q_2$, and $(X \col  P), y_2 \col  q_2 \vdash_{\mathcal{B}} \Psi \col  r$, where, without loss of generality, $y_1 \neq y_2$, we have
\[
\Psi[\Phi_2[\Phi_1/y_1]/y_2] = (\Psi[\Phi_2/y_2])[\Phi_1/y_1]
\]
\end{lemma}
\begin{proof}
By induction over the structure of $\Psi$.
\end{proof}

Functor categories of the form $\Set^\mathcal{W^{\op}}$, where 
the $(-)\op$ operation takes a category and reverses its morphisms,  
can be used to interpret proofs in NJ. The use of the $(-)\op$ operator 
is in line with the convention in the 
category theory community. The reason is that $\mathcal W$ embeds in 
$\Set^\mathcal{W^{\op}}$, and indeed $\Set^\mathcal{W^{\op}}$ can be 
characterized as a completion of $\mathcal W$. For more details see 
Section~\ref{sec:kripke}.

For suitable choices of $\mathcal{W}$, these functor categories can also 
be used to interpret derivations in a base. 

The basic idea is that the interpretation of an atomic proposition $p$ in 
$\Set^{\mathcal{W}^{\op}}$ is the functor whose value at world 
$(\mathcal B, (X\!:\!P))$ is the set of derivations of $p$ in $\mathcal B$ from 
hypotheses $(X\!:\!P)$. The action on morphisms of $\mathcal W$ is given by 
substitution. We use the following definition: 

\begin{definition}[Bases and Contexts] \label{def:bases-contexts}   
Define a category $\mathcal{W}$ as follows: 
\begin{itemize}
\item[--] Objects of $\mathcal{W}$ are pairs $(\mathcal{B},(X \col  P))$, 
where $\mathcal{B}$ is a base and $(X \col P)$ is a context; 
\item[--] A morphism from $(\mathcal{B}, (X \col  P))$ to $(\mathcal{C}, (Y \col 
Q))$ is given by an inclusion of the base $\mathcal{C}$ into $\mathcal{B}$ and
a set of derivations $ X: P \vdash_{\mathcal{B}}\Phi_i \col q_i$, where 
$Q = \{q_1,\ldots,q_m\}$. We write such a morphism as $(\Phi_1,\ldots,\Phi_m)$; 
\item[--] The identity morphism on $(X \col  P)$ is $(x_1, \ldots, x_n)$; 
\item[--] The composition of a morphism $(\Phi_1, \ldots, \Phi_m)$ from $(X \col 
P)$ to $(Y\col  Q)$ and a morphism $(\Psi_1, \ldots, \Psi_k)$ from $(Y\col  Q)$ 
to $(Z\col  R)$ is 
\[
(\Psi_1[\Phi_1/y_1, \ldots, \Phi_m/y_m], \ldots, \Psi_k[\Phi_1/y_1, \ldots, \Phi_m/y_m])
\] 
\end{itemize}
\end{definition} 
Lemma~\ref{lemmaAssoc} implies that composition is associative.

Now we extend the interpretation of atomic propositions $p$ to the interpretation of 
formulae using categorical products for conjunction and exponentials in functor categories 
for implication. We use products, exponentials, and a quantification over atoms to 
represent the form of disjunction employed by Sandqvist, which corresponds to its 
second-order definition (or, alternatively, the elimination rule for $\vee$ in NJ). 
See, for example, \cite{vanDalen}.

Key to understanding our categorical formulation of the exponentials in functor 
categories is the Yoneda lemma (see ~\cite{MM1992}): let $\mathcal{C}$ be a (locally 
small) category, $\Set$ be the category of sets, and $F \in [\mathcal{C}^{op},\Set]$ 
(the category of presheaves over $\mathcal{C}$); then, for each object 
$C$ of $\mathcal{C}$, with $h^C = \hom(-, C):\mathcal{C}\op\to\Set$, the natural transformations 
$h^C \to F$ are in bijection with the elements of $F(C)$ and this 
bijection is natural in $C$ and $F$. 

We will use the Yoneda lemma to calculate the values of 
certain functors between presheaf categories. These functors are defined as right adjoints.
Given a functor $G:\mathcal{C}\to\mathcal{D}$, a {\em right adjoint} 
to $G$ is a functor $H:\mathcal{D}\to\mathcal{C}$, whose defining characteristic is that
$\mathcal{D}(Gc,d)$ is naturally isomorphic to $\mathcal{C}(c,Hd)$.  
This notion has a dual. Given $H:\mathcal{D}\to\mathcal{C}$, a {\em left adjoint} is a 
functor $G:\mathcal{C}\to\mathcal{D}$, whose defining characteristic is that 
$\mathcal{D}(Gc,d)$ is naturally isomorphic to $\mathcal{C}(c,Hd)$.
The Yoneda lemma tells us that if $F$ is defined as a 
right adjoint, then its values are certain natural transformations between functors defined 
from the corresponding left adjoint.

\begin{lemma}\label{lemma:right-adjoints-presheaves}
Suppose $L:[\mathcal{C}\op,\Set] \to [\mathcal{D}\op,\Set]$ has a right 
adjoint denoted $R: [\mathcal{D}\op,\Set] \to [\mathcal{C}\op,\Set]$. Let  
$G:\mathcal{D}\op\to\Set$. Then for any $C\in\mathcal{C}$, $RG(C)$ 
can be taken to be the set of natural transformations $L(\hom(-,C))$ to $G$.
\end{lemma}

\begin{proof}
By the Yoneda lemma, $RG(C) \cong [\mathcal{C}\op,\Set](h^C,RG)$ 
and, by the adjunction, this is isomorphic to $[\mathcal{C}\op,\Set](L(h^C),G)$, 
the set of natural transformations from $L(\hom(-,C))$ to $G$. 
\end{proof}

Shortly, we shall apply this to characterize exponentials, with $L = (-)\times F$. 
In this case, we get that the value of $F\supset G$ at $C$ is the set of natural 
transformations from $\hom(-,C)\times F$ to $G$.

Products in the functor category $[\mathcal{W}\op,\Set]$ are given component-wise.  
For functors $F$ and $G$, we write $F \supset G$ for the exponential functor. 
This functor is defined as the right adjoint to $(-)\times F$, using the 
characterization given above:
\[
    (F \supset G)(\mathcal{B}, (X \col P)) = \mathop{Nat}(h^{(\mathcal{B}, (X : P))} \times F, G)
\]
where $\mathop{Nat}(\phi,\psi)$ denotes the set of natural transformations between the 
set-valued functors $\phi$ and $\psi$. 

We can now describe formally how disjunction is treated in terms of products and exponentials.  
Let $\mathcal{A}$ be the discrete category of atomic propositions (its set of objects is the set of atomic propositions and its only morphisms are identities). Define a functor 
$\Delta_{\mathcal{A}}: (\mathcal{W}^{op} \rightarrow \Set) \rightarrow ((\mathcal{W}^{op} \times \mathcal{A}) \rightarrow \Set)$ 
by mapping a functor $H:\mathcal{W}^{op} \rightarrow \Set$ to the functor 
$H': \mathcal{W}^{op} \times \mathcal{A} \rightarrow \Set$ such that 
$H' ((\mathcal{B}, (X\!:\!P)), p) = H((\mathcal{B}, (X\!:\!P)))$. Call the right adjoint 
to this functor $\forall_{\mathcal{A}}$. As a right adjoint, we can again 
characterize its values as natural transformations.  Specifically, if 
$K: (\mathcal{W}^{op} \times \mathcal{A}) \rightarrow \Set$, then 
$(\forall_{\mathcal{A}} K)(\mathcal{B}, (X\!:\!P))$ is the set of natural 
transformations from $\Delta_{\mathcal{A}}h^{(\mathcal{B}, (X : P))}$ 
to $K$. Since $\mathcal{A}$ is discrete, this amounts to giving a natural
transformation from $h^{(\mathcal{B}, (X : P))}$ to $K({-},p)$ for each 
atomic proposition $p$. 

Definition~\ref{def:int-func} defines the 
interpretation of formulae in presheaves. 
The key cases for our purposes are the base 
case, which captures the derivation of atoms in 
a base, the case for disjunction, as discussed at length elsewhere, and the case for falsity, 
which is interpreted as nullary disjunction.

\begin{definition}[Interpretation Functor] \label{def:int-func}   
Define a functor $\sem{\phi} \col  \mathcal{W}^{op} \rightarrow \Set$ by
induction over the structure of $\phi$ as follows:
\begin{itemize}
  \item[--] $\sem{p}(\mathcal{B}, (X \col  P))$ is the set of derivations
    $(X \col   P) \vdash_\mathcal{B} \Phi \col  p$. Any morphism
    $(\Phi_1, \ldots, \Phi_m)$ from $(\mathcal{B}, (X \col  P))$ to
    $(\mathcal{C}, (Y \col  Q))$ maps a derivation 
    $(Y\col  Q) \vdash_\mathcal{C} \Phi \col  p$, which is also a
    derivation $(Y\col  Q) \vdash_\mathcal{B} \Phi \col  p$, to the
    derivation $(X \col   P) \vdash_{\mathcal{B}} \Phi[\Phi_1/x_1,
    \ldots, \Phi_n/x_n] \col p$. 
    \item[--] $\sem{\phi \wedge \psi}$ is the product of the functors
      $\sem{\phi}$ and $\sem{\psi}$
\item[--] $\sem{\phi \supset \psi}$ is defined as $\sem{\phi} \supset \sem{\psi}$ 
\item[--] $\sem{\phi \vee \psi}$ is defined as follows: let $F = \sem\phi$, $G = \sem\psi$, and $K((\mathcal{B}, (X\!:\!P)), p) = (F \supset \sem{p}) \supset ((G \supset \sem{p}) \supset \sem{p})(\mathcal{B}, (X\!:\!P))$. This can be
extended to a functor $\mathcal{W}^{op} \times \mathcal{A} \rightarrow \Set$.  Then
$\sem{\phi \vee \psi}$ is defined as $\forall K$ 
\item[--] $\sem{\bot}$ is defined as follows: let $K((\mathcal{B}, (X\!:\!P)), p) = \sem{p}(\mathcal{B}, (X\!:\!P))$. This can be extended to a functor $\mathcal{W}^{op} \times \mathcal{A} \rightarrow \Set$.  Then
$\sem{\bot}$ is defined as $\forall_{\mathcal{A}} K$.
\end{itemize}
\end{definition} 

Note that in this interpretation, unlike in the basic interpretation of NJ proofs in 
bicartesian closed categories, disjunction does not correspond to coproduct in 
$[\mathcal{W}\op,\Set]$. We discuss this point in the sequel. 

In Section~\ref{sec:metatheory}, we establish soundness and completeness for our categorical 
formulation of Sandqvist's semantics. The proof of soundness uses the
existence of a natural transformation corresponding to $\Vdash$:
$\Gamma \Vdash \phi$ iff there exists a natural transformation from
$\sem{\Gamma}$ to $\sem\phi$. The proof of completeness uses a special
base, as in \cite{Sandqvist2015IL}, which is naturally extended via
$\sem{-}$ to the full consequence relation.

\begin{definition}[Categorical Base-extension Validity]\label{def:cat-base-val} 
We write $\Gamma \models \phi$ iff there is a natural transformation $\eta : \sem{\Gamma} \rightarrow \sem{\phi}$.    
\end{definition}

For each base $\mathcal{B}$ we show soundness and completeness
relative to this base by considering a suitable full subcategory.
Soundness and completeness then arise by considering the empty base.

For any base $\mathcal{B}$, we write
$\mathcal{W}_{\mathcal{B}}$ for the full subcategory of
$\mathcal{W}$ where the objects are pairs $(\mathcal{C}, (Y:Q))$ such
that $\mathcal{B} \subseteq \mathcal{C}$.
For any full subcategory
${\mathcal{W}'}$ of $\mathcal{W}$, we write $\semwp{-}$ for
  the functor $\sem{-}$ restricted to objects and morphisms in
  $\mathcal{W}'$.

 In Lemma~\ref{lem:algSoundInd}, we assume, without loss of generality, that all 
  $\phi^\flat$ in $\mathcal{N}_{\Delta^\flat}$ and the 
  $\flatOp{\phi_1} , \ldots , \flatOp{\phi_n}$ are distinct from the atoms in $\mathcal{B}$

\begin{lemma} \label{lem:algSoundInd}
    Let $\Gamma = \phi_1, \ldots, \phi_n$.
 Let $\mathcal{B'} =
  \mathcal{B} \cup \mathcal{N}_{\Delta^\flat} \cup \{\Rightarrow
  \flatOp{\phi_1}, \ldots, \Rightarrow \flatOp{\phi_n}\}$. 
Let $\mathcal{W'} = \mathcal{W}_{\mathcal{B}}$.
If  $P \vdash_{\mathcal{B'}} p$, then there exists a natural
transformation $\eta \col \semwp{\Gamma} \times \semwp{\natOp{P}}
\rightarrow \semwp{\natOp{p}}$.
\end{lemma}
\begin{proof}
  The proof proceeds by induction over the structure of the derivation of $P \vdash_{\mathcal{B'}} p$
  structured by cases on the last rule used in the derivation. For the
  case of an $\vee$-elimination rule, we use again an induction over the structure of the
 $\phi$. Because the semantics of $\vee$ does not use a coproduct, we
 cannot use a universal property in this case and by using this
 induction consider in the
 proof only all objects which are in the range of the semantics $\sem{-}$.

  \medskip

  \noindent Rules in $\{\Rightarrow \flatOp{\phi_1}, \ldots, \Rightarrow \flatOp{\phi_n}\}$:
  \begin{itemize}[label=--,leftmargin=5mm,nosep]
      \item Suppose the last rule is $\Rightarrow  \flatOp{\phi}$. Then $\natOp{\flatOp{\phi}}=\phi$ and the
        projection from $\semwp{\Gamma} \times \semwp{\natOp{P}}$ to
          $\semwp{\phi}$ yields the claim.
  \end{itemize}
  \noindent Rules in $\mathcal{B}$:
  \begin{itemize}[label=--,leftmargin=5mm,nosep]
    \item Suppose the derivation is $P,p \vdash_{\mathcal{B'}} p$.
        The projection from $\semwp{\natOp{P, p}}$ to
        $\semwp{\natOp{p}}$ has the desired properties.
      \item
        Suppose the last rule of the derivation is an application of
        the rule $\mathcal{R} = ((P_1 \Rightarrow p_1), \ldots, (P_n \Rightarrow
        p_n))\Rightarrow p$. By the induction hypothesis, there are natural
        transformations $\eta_i$ from $\semwp{\Gamma} \times \semwp{\natOp{P}, P_i}$ to
        $\semwp{p_i}$ . Hence there are
        also natural transformations $\eta_i'$ from
        $\semwp{\Gamma} \times \semwp{\natOp{P}}$ to $\semwp{P_i
          \supset p_i}$. Now consider any object of ${\mathcal{W}'}$. This is a pair $(\mathcal{C},(Y\col Q))$ 
        such that $\mathcal{B} \subseteq \mathcal{C}$. 
        Let $f$ be an element  of $\semwp{\Gamma}(\mathcal{C}, (Y \col
        Q))$
        and let $g$ be an element of $\semwp{\natOp{P}}(\mathcal{C}, (Y \col
        Q))$. Hence
          ${\eta_i}'_{(\mathcal{C}, (Y \col\, Q))}(f, g)$ is a natural transformation
          $\kappa_i$ from $Hom(-, (\mathcal{C}, (Y \col\, Q)) \times \semwp{P_i}$
            to $\semwp{p_i}$.  Let $\pi_i$ the projection from
            $(Y \col Q), (Z_i\col P_i)$
            to $P_i$. By definition, $\kappa_i(Id, \pi_i)$ is a
            derivation $\Phi_i \col (Y \col Q),(Z \col P_i), \vdash_{\mathcal{C}}
              p_i$. Hence $\Phi_{\mathcal{R}}(\Phi_1, \ldots, \Phi_n)$
                is a derivation of $Y \col Q \vdash_{\mathcal{C}} p$.
  \end{itemize}
  \noindent Rules in $\mathcal{N}_{\Delta^\flat}$:
  \begin{itemize}[label=--,leftmargin=5mm,nosep]
          \item Suppose the last rule is $\flatOp{\psi_1}, \flatOp{\psi_2}
            \Rightarrow \flatOp{(\psi_1 \wedge \psi_2)}$. By the induction
            hypothesis, there are two natural transformations $\eta_1$
            and $\eta_2$, where $\eta_i$ is a natural transformation 
            from $\semwp{\Gamma}\times \semwp{\natOp{P}}$ to
          $\semwp{\psi_i}$. Hence $\produ{\eta_1,
            \eta_2}$ is a natural transformation from
          $\semwp{\Gamma}\times \semwp{\natOp{P}}$ to
          $\semwp{\psi_1 \wedge \psi_2}$.
\item Suppose the last rule is $\flatOp{(\psi_1 \wedge \psi_2)} \Rightarrow
  \flatOp{\psi_1}$. By the induction hypothesis, there exists a natural transformation
  $\eta$ from $\semwp{\Gamma} \times \semwp{\natOp{P}}$ to
          $\semwp{\psi_1 \wedge \psi_2}$. Now compose
          this natural transformation with the projection from $\semwp{\psi_1 \wedge
            \psi_2}$ to $\semwp{\psi_1}$.
\item Suppose the last rule is $(\flatOp{\psi_1} \Rightarrow
  \flatOp{\psi_2}) \Rightarrow \flatOp{(\psi_1 \supset\psi_2)}$. By
  assumption, there is a natural transformation from
  $\semwp{\Gamma}\times \semwp{\natOp{P}} \times
  \semwp{\psi_1}$ to $\semwp{\psi_2}$. Hence there is also a natural
  transformation from $\semwp{\Gamma} \times \semwp{\natOp{P}}$ to
$\semwp{\psi_1 \supset \psi_2}$.
  \item Suppose the last rule is $\flatOp{\psi_1}, \flatOp{(\psi_1 \supset
      \psi_2)} \Rightarrow \flatOp{\psi_2}$. By assumption, we have
    natural transformations 
    $\eta_1$ and $\eta_2$ from $\semwp{\Gamma}\times
    \sem{\natOp{P}}$ to $\semwp{\psi_1}$ and $\semwp{\psi_1 \supset
    \psi_2}$ respectively. By the definition of function spaces for
  functor categories, there exists also a natural transformation from
  $\semwp{\Gamma}\times \sem{\natOp{P}}$ to $\semwp{\psi_2}$.
          \item Suppose the last rule is $\flatOp{\psi_1} \Rightarrow
            \flatOp{(\psi_1 \vee \psi_2)}$. By definition of the
            interpretation of disjunction there is a
            natural transformation $\kappa$ from $\semwp{\psi_1}$ to
            $\semwp{\psi_1 \vee \psi_2}$.
            By the induction hypothesis, there
            is a natural transformation $\eta$ from $\semwp{\Gamma}
            \times \semwp{\natOp{P}}$ to $\semwp{\psi_1}$. $\kappa \circ
            \eta$ is therefore a natural transformation from  $\semwp{\Gamma}\times \sem{\natOp{P}}$ to
            $\semwp{\psi_1\vee\psi_2}$.
\item
  Suppose the last rule is $\flatOp{(\psi_1 \vee \psi_2)}, (\flatOp{\psi_1}
  \Rightarrow p), (\flatOp{\psi_2} \Rightarrow p) \Rightarrow
  p$. By the induction hypothesis, there are natural transformations
  $\eta_{\psi_1 \vee \psi_2} $ from $\sem{\Gamma}\times \semwp{\natOp{P}}$ to
  $\semwp{\psi_1\vee\psi_2}$, $\eta_{\psi_1}$ from $\semwp{\Gamma}\times \sem{\natOp{P}} \times
  \semwp{\psi_1}$ to $\semwp{\natOp{p}}$ and
  $\eta_{\psi_2}$ from $\semwp{\Gamma}\times \semwp{\natOp{P}}\times
  \semwp{\psi_2}$ to $\semwp{\natOp{p}}$.  
Now we use an induction over the structure of $\natOp{p}$.
\begin{itemize}
\item $\natOp{p} = p$:      By the
  definition of the interpretation of disjunction there is
  also a natural transformation from $\semwp{\psi_1 \vee \psi_2}  \times \semwp{\psi_1 \supset p}\times \semwp{\psi_2
    \supset p}$ to $\semwp{p}$. Hence
  there is also a natural transformation from $\semwp{\Gamma}\times
  \semwp{\natOp{P}}$ to   $\semwp{p}$. 
\item $\natOp{p} = \phi_1 \wedge \phi_2$:
We also have $P, \flatOp{\psi_1} \vdash_{\mathcal{B'}} \flatOp{\phi_i}$ and $P
, \flatOp{\psi_2} \vdash_{\mathcal{B'}} \flatOp{\phi_i}$ for $i = 1,
2$. By the induction hypothesis, there 
are natural transformations $\eta_i$ from $\semwp{\Gamma}\times
  \sem{\natOp{P}}$ to   $\semwp{\phi_i}$. The natural transformation
  $\produ{\eta_1, \eta_2}$ is a natural transformation from
  $\semwp{\Gamma}\times \sem{\natOp{P}}$ to  $\semwp{\phi_1 \wedge
    \phi_2}$.
\item $\natOp{p} = \phi_1 \supset \phi_2$:
We also have $P, \flatOp{\phi_1}, \flatOp{\psi_i}  \vdash_{\mathcal{B'}}
\flatOp{\phi_2}$  and $P, \flatOp{\phi_1} \vdash_{\mathcal{B'}}
\flatOp{(\psi_1 \vee \psi_2)}$. By the induction hypothesis, there is a
natural transformation from $\semwp{\Gamma}\times \semwp{\natOp{P}}
\times \semwp{\phi_1}$ to $\semwp{\phi_2}$. Hence there is also a
natural transformation from $\semwp{\Gamma}\times \semwp{\natOp{P}}$ to
$\semwp{\phi_1 \supset \phi_2}$.
\item $\natOp{p} = \phi_1 \vee \phi_2$: 
      For any atom $q$ and any $F_1 \colon \semwp{\phi_1} \supset
      \semwp{q}$, $F_2 \colon \semwp{\phi_2} \supset
      \semwp{q}$, we define a natural transformation $\mu^{\psi_1} \colon
      \semwp{\Gamma} \times \semwp{\natOp{P}} \times \semwp{\psi_1} \rightarrow \semwp{q}$ by
      $\mu^{\psi_1}(\gamma, f, t) = \eta_{\psi_1}(\gamma, f, t)pF_1F_2$. We
      similarly define a natural transformation $\mu^{\psi_2} \colon
   \semwp{\Gamma} \times \semwp{\natOp{P}} \times \semwp{\psi_2} \rightarrow \semwp{q}$ by
      $\mu^{\psi_2}(\gamma, f, t) = \eta_{\psi_2}(\gamma, f, t)pF_1F_2$. 
      Now we define, using an informal $\lambda$-calculus notation a natural transformation $\eta \colon
      \semwp{\Gamma} \times \semwp{\natOp{P}}
      \rightarrow \semwp{\phi_1 \vee \phi_2}$ by
      \[
      \begin{array}{lcl}
      \eta(\gamma, f) & = & \lambda q \,.\, \lambda F_1 \colon\sem{\phi_1} \supset
      \semwp{q} \,.\, \lambda F_2 \colon \sem{\phi_2} \supset
      \semwp{q} \,.\, \\ 
      & & \eta_{\psi_1 \vee \psi_2}(\gamma,f) p
      (\mathop{Cur}(\mu^\phi)(\gamma, f))(\mathop{Cur}(\mu^\psi)(\gamma,f))
      \end{array}
      \]
\end{itemize}
    \end{itemize}
\end{proof}

\begin{lemma}[Algebraic Soundness]\label{lem:alg-sound}
  Suppose $\Gamma \Vdash_{\mathcal{B}} \phi$. Let $\mathcal{W}'$ be the category
  $\mathcal{W}_{\mathcal{B}}$. Then there exists a natural
  transformation $\eta_{\mathcal{B}} \col \semwp{\Gamma} \rightarrow \semwp{\phi}$.
\end{lemma}
\begin{proof}
  Let $\Gamma = \phi_1, \ldots, \phi_n$.
 Let $\mathcal{B'} =
  \mathcal{B} \cup \mathcal{N}_{\Delta^\flat} \cup \{\Rightarrow
  \flatOp{\phi_1}, \ldots, \Rightarrow \flatOp{\phi_n}\}$.
Using the proof of Sandqvist's theorem~5.1, we have
$\Vdash_{\mathcal{B'}} \Gamma$. Hence we have by assumption also
  $\Vdash_{\mathcal{B'}} \phi$, and, using the proof of Sandqvist's
    theorem~5.1 again, we have $\Vdash_{\mathcal{B'}} \flatOp{\phi}$.
Lemma~\ref{lem:algSoundInd} now yields the claim.
    \end{proof}

    Now we turn to completeness, which is formulated as the converse of soundness. 
    
 
\begin{lemma}[Algebraic Completeness] \label{lem:alg-complete}
Consider any base $\mathcal{B}$. Let $\mathcal{W}'$ be the category
  $\mathcal{W}_{\mathcal{B}}$.  If there exists a natural
  transformation $\eta_{\mathcal{B}} \col \semwp{\Gamma} \rightarrow \semwp{\phi}$,
  then $\Gamma \Vdash_{\mathcal{B}} \phi$.

\end{lemma}
\begin{proof}
The proof proceeds by induction over the structure of $\Gamma$ and $\phi$. 
\begin{description}[leftmargin=5mm,nosep]
\item[$\Gamma = P, \phi = p$] Let $id$ be the identity on $P$. By
  definition, $\eta_{\mathcal{B},(\mathcal{B}, (X:P)) }(id)$ is a derivation --- i.e., a proof 
according to Definition~\ref{def:validity-in-a-base} --- $P \Vdash_{\mathcal{B}} p$. 
\item[$\phi = \phi_1 \wedge \phi_2$] By definition, there exists
  natural transformations $\eta_i \colon \sem{\Gamma}^{\mathcal{W'}}
  \rightarrow \sem{\phi_i}^{\mathcal{W'}}$, and hence, by the induction hypothesis, we have $\Gamma \Vdash_{\mathcal{B}} \phi_i$ and therefore $\Gamma \Vdash_{\mathcal{B}} \phi_1 \wedge \phi_2$.
\item[$\Gamma = \Delta, \phi_1 \wedge \phi_2$] By the induction hypothesis, there exists a derivation $\Delta, \phi_1, \phi_2 \Vdash_{\mathcal{B}} \phi$, which is the required derivation, as $\sem{\Delta, \phi_1 \wedge \phi_2} = \sem{\Delta, \phi_1, \phi_2}$.
\item[$\phi = \phi_1 \supset \phi_2$] By the definition of $\supset$
  as a right adjoint, there exists a natural transformation  $\eta' \colon \sem{\Gamma, \phi_1}^{\mathcal{W'}} \rightarrow \sem{\phi}^{\mathcal{W'}}$. By the induction hypothesis, there exists a derivation $\Gamma, \phi_1\Vdash_{\mathcal{B}} \phi_2$, and therefore also a derivation  $\Gamma \Vdash_{\mathcal{B}} \phi_1 \supset \phi_2$.
\item[$\Gamma = \Delta, \phi_1 \supset \phi_2$] Suppose
  $\Vdash_{\mathcal{C}} \Gamma$. By algebraic soundness, for all
  $\mathcal{D}$ such that $\mathcal{C} \subseteq \mathcal{D}$ and $P$ there are
  elements $\delta$ of $\sem{\Delta}(\mathcal{D}, (X \col P))$ and $f$ of
  $\sem{\phi_1 \supset \phi_2}(\mathcal{D},
  (X \col P))$. $\eta_{\mathcal{B}, (\mathcal{D}, (X:P))}(\delta, f)$ is an element of
  $\sem{\phi}(\mathcal{D}, (X\col P))$. This assignment is natural in
  $(\mathcal{D}, (X:P))$. By the induction hypothesis, we have
  $\Vdash_{\mathcal{C}} \phi$.
\item[$\phi = \phi_1 \vee \phi_2$]  By the definition of $\sem{-}$,
  for all atoms $p$ there exists a natural transformation
  $\eta_{\mathcal{B}}$ from $\sem{\Gamma}^{\mathcal{W'}}$ to
  $\sem{\phi_1 \vee \phi_2}^{\mathcal{W'}}$ iff for all atoms $p$
  there exists a natural transformation $\eta_{p, \mathcal{B}}$ from
  $\sem{\Gamma}^{\mathcal{W}'}$ to $\sem{(\phi_1  \supset p) \supset
    (\phi_2 \supset p) \supset p}^{\mathcal{W}'}$. Hence the induction
  hypothesis yields $\Gamma \Vdash _{\mathcal{B}} (\phi_1 \supset p)
  \supset (\phi_2 \supset p) \supset p $. Hence, by definition,
  $\Gamma \Vdash_{\mathcal{B}} \phi_1 \vee \phi_2$.
\item[$\Gamma = \Delta, \phi_1 \vee \phi_2$]
By induction over the structure of $\phi$. 
\begin{description}[leftmargin=5mm,nosep]
\item[$\phi = p$] 
Suppose $\Vdash_{\mathcal{B}} \Gamma$.  By algebraic soundness, there
exists a natural transformation $\eta': \sem{\top}^{\mathcal{W'}}
\rightarrow \sem{\Gamma}^{\mathcal{W'}}$.  By definition, 
$\eta_{\mathcal{B}} \circ \eta'$ is a natural tranformation from $\sem{\top}^{\mathcal{W'}} $ to $\sem{p}^{\mathcal{W}'}$. As $\sem{\top}(\mathcal{B}, (-:-)) $ is non-empty, $\sem{p}(\mathcal{B}, (-:-))$ is non-empty as well, and hence by the definition of $\sem{-}$, there exists a derivation $\Vdash_{\mathcal{B}} p$.
\item [$\phi = \psi_1 \wedge \psi_2$] By definition, there exists
  natural transformations $\eta_1$ and $\eta_2$ such that $\eta = \produ{\eta_1, \eta_2}$. By the induction hypothesis, there are derivations $\Gamma \Vdash_{\mathcal{B}} \psi_1$ and $\Gamma \Vdash_{\mathcal{B}} \psi_2$. Therefore we also have $\Gamma \Vdash_{\mathcal{B}} \psi_1 \wedge \psi_2$.
\item[$\phi = \psi_1 \supset \psi_2$] By definition, there exists also
  a natural transformation $\eta'_{\mathcal{B}}$  from $\sem{\Gamma, \psi_1}^{\mathcal{W}'}$ to $\sem{\psi_2}^{\mathcal{W}'}$. By the induction hypothesis, there exists a derivation $\Gamma, \psi_1 \Vdash_{\mathcal{B}} \psi_2$. Hence there is also a derivation $\Gamma \Vdash_{\mathcal{B}} \psi_1 \supset \psi_2$.
\item[$\phi = \psi_1 \vee\psi_2$] Consider any atom $p$. By definition
  of $\sem{-}$, there exists a natural transformation $\eta'$ from $\sem{\Gamma}^{\mathcal{W}'}$ to $\sem{(\psi_1 \supset p) \supset (\psi_2 \supset p) \supset p}^{\mathcal{W}'}$. By the induction hypothesis, there is a derivation  $\Gamma \Vdash_{\mathcal{B}} (\psi_1 \supset p) \supset (\psi_2 \supset p) \supset p$. Hence there is also a derivation $\Gamma \Vdash_{\mathcal{B}} \psi_1 \vee\psi_2$.
 \end{description}
\end{description}
\end{proof}

\begin{proposition}[Equivalence] \label{prop:equiv} 
$\Gamma \Vdash \phi$ iff $\Gamma \models \phi$. 
\end{proposition}
\begin{proof}
Direct consequence of Lemmas~\ref{lem:alg-sound} and~\ref{lem:alg-complete}.
\end{proof}

\section{Relation to Kripke, Presheaf, and Sheaf Models}\label{sec:kripke}

The previous section has set out the details of a categorical account of Sandqvist's proof-theoretic semantics. In this section, we situate 
that account with reference to classical Kripke models, to the algebraic treatment of intuitionistic logic via complete Heyting algebras, and to the 
categorical treatment of higher order type theory via categories of presheaves and sheaves. 

The treatment of conjunction and implication is standard.
The treatment of disjunction, however, is not, and the close analysis of this exposes that there is a distinction between the proof-insensitive algebraic 
semantics and the proof-sensitive categorical semantics. In the algebraic semantics,  the soundness of Sandqvist's definition of disjunction can be 
established by exposing it as the natural disjunction in a sublocale (quotient complete Heyting algebra) of the locale obtained from the standard Kripke 
model (Definition \ref{def:nucleusK}, Lemma \ref{lemma:sublocale-interpretation}). The analogue in the categorical setting is sheaves for the 
corresponding (Grothendieck or Joyal-Tierney) topology. However, there is no real reason for the interpretations of atomic propositions to be sheaves, and 
the definition of disjunction is not the categorical coproduct. 

The soundness of the interpretation instead comes from a structural induction, not a 
property of arbitrary sheaves (Lemmas \ref{lem:disjunction-support-closure}, \ref{lem:alg-sound} and Proposition \ref{prop:presheaf-sound}). 
This is unusual in categorical interpretations of type theory, and so noteworthy. 


Classical Kripke models are grounded in a set of worlds, $\kripkeWorlds$.
For each atomic predicate $p$, we are told whether $p$ holds in world $w$:
$w\kripkeForces p$. When interpreting intuitionistic logic, the set of worlds 
can be thought of as possible states of knowledge. This set is partially 
ordered with the ordering representing increasing knowledge. This 
viewpoint gives rise to and is reflected in the monotonicity property: 
if $w\leq w'$ and $w\kripkeForces p$, then $w'\kripkeForces p$. 

The definition of validity is extended to arbitrary propositions by 
structural induction: 
\[
\begin{array}{r@{\quad}r@{\quad}c@{\quad}l} 
(\supset) & \mbox{$w\kripkeForces \phi \supset \psi$} 
& \mbox{iff} & \mbox{for all $w\leq w'$, if $w'\kripkeForces\phi$, then $w'\kripkeForces \psi$} \\ 
(\wedge) & \mbox{$w\kripkeForces \phi \wedge \psi$} & \mbox{iff} & \mbox{$w\kripkeForces \phi$ and $w\kripkeForces \psi$} \\     
 (\bot) & \mbox{$w \kripkeForces \bot$} & \mbox{never} & \\
 (\vee) & \mbox{$w \kripkeForces \phi\vee\psi$} & \mbox{iff} & \mbox{$w\kripkeForces\phi$ or $w\kripkeForces\psi$}\\
\end{array}
\]
The monotonicity property then extends to all formulae. The 
quantification over extensions in the clause for $(\supset)$ is required 
in order to ensure this. 

Base-extension semantics relates closely to Kripke models. In the simplest formulation (ours is a little
more complex), the set of 
worlds, $\kripkeWorlds$ is the set of bases being used, and the partial 
order is simply inclusion of bases. We define 
${\mathcal B}\kripkeForces p$ iff $\vdash_{\mathcal B} p$. 
This satisfies the monotonicity property because if 
${\mathcal B}\subseteq {\mathcal C}$ and $\vdash_{\mathcal B} p$, 
then $\vdash_{\mathcal C} p$. There are two main approaches to extending 
this definition for atomic formulas to general ones: 
\begin{itemize}
    \item[-] the approach of Schroeder-Heister et {al.} ({cf.} \cite{PiechaS19a}), which 
            follows exactly the Kripke definitions
    \item[-] the approach of Sandqvist \cite{Sandqvist2015IL}, which follows the 
            Kripke definitions for implication and conjunction, but not 
            disjunction or false. 
\end{itemize}

However, Kripke models also have a more algebraic interpretation.

\begin{lemma}
In any Kripke model, the interpretation of an arbitrary proposition
$\phi$ satisfies the monotonicity property: if $w\leq w'$ and 
$w\kripkeForces \phi$, then $w'\kripkeForces \phi$. 
\end{lemma}

As a consequence, the interpretation of any proposition $\phi$ also 
satisfies the monotonicity property and so is an upwards-closed subset of $\kripkeWorlds$. 

\begin{definition}\label{def:kripkeOmega}
Let $\Omega$ be the set of upwards-closed subsets of $\kripkeWorlds$; that is, 
\[
\Omega = \{ U\subseteq\kripkeWorlds \mid \mbox{\rm if ${\mathcal B}\in U$ 
    and ${\mathcal B}\subseteq{\mathcal C}$, then ${\mathcal C}\in U$} \} 
\]
$\Omega$ is partially ordered by inclusion.
\end{definition}

\begin{lemma}
$\Omega$ is a (complete) Heyting algebra, and the interpretation of the 
logical connectives in that algebra coincides with that in the Kripke model. 
\end{lemma}

Moreover, if $w\in\kripkeWorlds$, then 
$w\up = \{ w' \in \kripkeWorlds \mid w\leq w' \}$ is in $\Omega$. 
However, if $w_0\leq w_1$ then $w_{1}\up \subseteq w_{0}\up$. So it is 
$\kripkeWorlds\op$ that embeds in $\Omega$, not $\kripkeWorlds$. Indeed, 
$\Omega$ can be seen as a free completion under arbitrary disjunctions of 
$\kripkeWorlds\op$. This means that Kripke models can be seen as simply 
a convenient way of presenting a complete Heyting algebra. 

The standard Kripke account gives an interpretation of validity, but not 
really of proof, or justification. In order to obtain that, we need to 
move to a setting which can distinguish between proofs that are different. One 
way of doing that is to use presheaves. 

\begin{definition}
If $\mathbf C$ is a category, then a {\em presheaf on ${\mathbf C}$} is 
a functor $F:{\mathbf C}\op \to\Set$, and a {\em morphism of presheaves} 
is a natural transformation between the functors. 
\end{definition}

Categories of presheaves generalize Kripke models. To cast a Kripke model 
in terms of presheaves, let ${\mathbf C} = \kripkeWorlds\op$, so 
${\mathbf C}\op = \kripkeWorlds$. As a result the presheaf category
corresponding to the Kripke model is actually $\Set^{\kripkeWorlds}$. 
Here we take the standard interpretation 
of a partial order as a category: the elements of the partial order form 
the objects of the category, and the hom-set $\kripkeWorlds (w_0,w_1)$ 
contains a single element if $w_o\leq w_1$ and is empty otherwise. 

In making the connection between Kripke and presheaves, a simple but 
important presheaf is the presheaf $1$, defined by 
$1(w) = \{*\}$. If $\phi$ is a sub-presheaf of this, then each $\phi(w)$ 
will either be $\emptyset$ or $\{*\}$. And if $w_0\leq w_1$, then there 
is a function $\phi(w_0) \to \phi(w_1)$. As a result, if $\phi(w_0)$ is non-empty, 
then so is $\phi(w_1)$. Hence $\phi$ corresponds to an upwards-closed 
subset of $\kripkeWorlds$. Conversely, if $U$ is an upwards-closed subset 
of $\kripkeWorlds$, then we can define $\chi_U : \kripkeWorlds\to\Set$ by 
\begin{itemize}
    \item[--] $\chi_U (w) = \{*\}$ if $w\in U$
    \item[--] $\chi_U (w) = \emptyset$ otherwise
\end{itemize}
$\chi_U$ is a subfunctor of $1$. For the case of a category of presheaves 
corresponding to a Kripke model, this demonstrates an order-preserving  
correspondence between the possible interpretations of propositions in the 
Kripke model and the subfunctors of $1$ in the presheaf category.
We have established:

\begin{lemma}
    There is an order-preserving isomorphism between
    \begin{itemize}
        \item[-] $\Omega$, the lattice of upwards-closed subsets of $\kripkeWorlds$, and
        \item[-] the lattice of subfunctors of $1$ in the presheaf topos $\Set^{\kripkeWorlds}$.
    \end{itemize}
\end{lemma}

This means that either of these lattices can be used 
interchangeably as the target of a validity-based semantics for 
intuitionistic propositional logic. 

However, categories of presheaves include functors targeting sets with 
more than one element. They do not just provide models of intuitionistic 
propositional logic, they form models of simple intuitionistic type theory.
They carry interpretations of type-forming operations including products
and sums. 

In the following let $F$ and $G$ be two functors $\mathbf{C}\op\to\Set$.
\begin{itemize}
    \item[-] The product of $F$ and $G$ in the category of presheaves is the functor $F\times G$ defined on objects by $F\times G (A) = F(A)\times_{\Set} G(A)$, where the action on morphisms is component-wise. 
    \item[-] The categorical function space $[F\to G]$ is defined on objects as follows:
    $[F\to G](A)$ is the set of natural transformations from $F\times \mathbf{C}({-},A)$ to $G$. 
    The action on morphisms is obtained via composition. 
    \item[-] The coproduct of $F$ and $G$ is the functor $(F+G)(A) = F(A) +_{\Set} G(A)$, where $+_{\Set}$ is the coproduct (i.e., disjoint union) in $\Set$.
\end{itemize}

If $\mathbf{C} = \kripkeWorlds\op$ is a partial order, and $F$ and $G$ are 
subfunctors of $1$, then these definitions correspond to the 
interpretations of $\wedge$, $\supset$, and $\vee$ in the corresponding 
Kripke model. 

Moreover, there are links to the categorical interpretation in 
Section~\ref{sec:categorical}. The constructions given there are internal 
constructions in the topos of presheaves: $\mathcal{W}\op\to\Set$.  We 
recall from Definition \ref{def:int-func}, that the interpretation of 
an atom $p$ is a functor $\mathcal{W}\op\to\Set$, and that for arbitrary 
propositions, $\phi$ and $\psi$, the interpretations of $\phi\wedge\psi$ and $\phi\supset\psi$ are defined to be $\sem{\phi}\times\sem{\psi}$ and 
$\sem{\phi}\supset\sem{\psi}$ respectively. 

\begin{lemma} For atoms, $p$ and arbitrary propositions $\phi$ and 
$\psi$, the interpretations defined in Definition \ref{def:int-func} 
have the following properties: 
\begin{itemize}[label=--,leftmargin=5mm,nosep]
    \item the interpretation of an atom $p$ is a functor 
    $\sem{p}:\mathcal{W}\op\to\Set$
    \item the interpretation of $\phi\wedge\psi$ is $\sem{\phi}\times\sem{\psi}$
    \item the interpretation of $\phi\supset\psi$ is $\sem{\phi}\to\sem{\psi}$.
\end{itemize}
\end{lemma}

However, the interpretation of $\phi\vee\psi$ is not the coproduct 
$\sem{\phi}+\sem{\psi}$ in the presheaf topos 
$\mathcal{W}\op\to\Set$. 
Instead it is the interpretation of the second order formula:
\[\forall p. (\sem{\phi}\to \sem{p}) \to (\sem{\psi}\to \sem{p}) \to \sem{p} \]
There are two equivalent ways of formalizing this intuition. 
\begin{itemize}[label=--,leftmargin=5mm,nosep]
    \item We can take the collection of atomic propositions to be an external set, and the quantification as the conjunction of all the 
    $(\sem{\phi}\to \sem{p}) \to (\sem{\phi}\to \sem{p}) \to \sem{p}$, using the fact that the lattice of propositions has infinite meets.
    \item We can internalize the collection of atomic propositions as
    the constant presheaf $\Delta\setAtoms$, where $(\Delta\setAtoms)(w) = \setAtoms$. We interpret $(\sem{\phi}\to \sem{p}) \to (\sem{\phi}\to \sem{p}) \to \sem{p}$ as an internal family of propositions indexed by $\Delta\setAtoms$, and interpret the $\forall p$ as the internal quantification in the logic of the topos, which in this case is the product: $\prod_p. (\sem{\phi}\to \sem{p}) \to (\sem{\psi}\to \sem{p}) \to \sem{p}$.
\end{itemize}
The second of these gives rise to the construction detailed in Section \ref{sec:categorical}. 

It follows that the structure we have defined in 
Section~\ref{sec:categorical} is naturally a model of 
$F_{at}$, a restriction of the second-order lambda calculus with function types and polymorphism restricted 
to be over a collection of atomic types. The restriction of $\forall$ to a predefined set means that 
$F_{at}$ is predicative, in contrast to the full 
System $F$ \cite{Girard71,GLT89}, and it further means that it can be interpreted in a broadly set-theoretic setting where the 
universal types are interpreted as products over the set of atomic types. This is exactly what our 
categorical model is doing. There is therefore a strong link between the structures here and those 
investigated extensively by Ferreira et al. \cite{F2006,FF2013a,FF2013b,FF2015,FES2020,PTP2022}.  However, Ferreira's treatment is essentially syntactic, while ours 
is strongly semantic. In particular, the types of $F_{at}$ exist in a wider context. However, our proof of 
Lemma \ref{lem:disjunction-support-closure} can be seen as a semantic equivalent of Ferreira's `instantiation 
overflow' for the case of disjunction. 

This interpretation supports the introduction rules for disjunction:
for any objects $A$ and $B$ of the topos there are morphisms 
$A\to \forall p. (A\to \sem{p}) \to (B\to \sem{p}) \to \sem{p}$ and 
$B\to \forall p. (A\to \sem{p}) \to (B\to \sem{p}) \to \sem{p}$. 
But there is not a similarly general interpretation of the elimination 
rule. Instead, we get an interpretation where one of the variables is 
restricted to range over interpretations of formulae rather than arbitrary
objects of the topos. 

Consider the disjunction elimination rule: 
\[\infer[\vee{E}]{\Gamma \vdash \chi}
            {\Gamma \vdash \phi \vee \psi & \Gamma, \phi \vdash \chi &
              \Gamma, \psi \vdash \chi}\]

We want to encode this rule in terms of the structure of the topos. To 
do that, we consider an interpretation where the metavariables $\phi$, $\psi$, 
and $\chi$ represent objects of the topos rather than formulae. 
Replacing $\phi$, $\psi$, and $\chi$ by $A$, $B$, and $C$, we see that we want a morphism 
\[
    \sigma(A,B)\times C^A \times C^B \to C
\]
where $\sigma$ is a binary operation on objects of the topos giving the semantics of disjunction. 

Specifically, Sandqvist's semantics uses 
$\sigma(A,B) = \forall p.  (A\to \sem{p}) \to (B\to \sem{p}) \to \sem{p}$, 
which is defined in any sheaf topos. For this interpretation, if $\chi$ is interpreted by an arbitrary 
object, then the rule is not sound.  However, if the object that interprets $\chi$ is constructed 
propositionally from the interpretations of atoms, and so is the interpretation of a proposition, 
then soundness is recovered.  

\begin{definition}\label{def:supports-disjunction}
We say that a type constructor $\sigma(A,B)$ supports disjunction elimination for an object $C$ of the topos if there is a morphism
\[\sigma(A,B)\times [A\to C] \times [B\to C] \longrightarrow C\]
\end{definition}

The objects for which a constructor $\sigma$ supports disjunction elimination have good closure properties. 

\begin{lemma}\label{lem:disjunction-support-closure} \mbox{ }
\begin{enumerate}
\item $\sigma$ always supports disjunction elimination for the terminal object $1$; that is, when $C$ is $1$, the constant functor taking value $\{*\}$. 
\item If $\sigma$ supports disjunction elimination for $C_0$ and $C_1$, then $\sigma$ supports disjunction elimination for $C_0\times C_1$. 
\item If $\sigma$ supports disjunction elimination for $C$, then $\sigma$ supports disjunction elimination for $[D\to C]$. 
\item If $\sigma$ supports disjunction elimination for any $C_x$ where $x\in X$, then $\sigma$ supports disjunction elimination for $\prod_{x\in X} C_x$ (and a similar but more complex statement for internal products in the topos). 
\end{enumerate}
\end{lemma}

\begin{proof}
The proofs are all straightforward manipulation of the cartesian closed structure of the topos. 
\begin{enumerate}
    \item We require a morphism $\sigma(A,B)\times [A\to 1] \times [B\to 1] \to 1$, which exists uniquely because $1$ is terminal. 
    \item\label{lem:case-prod} We require a morphism 
    \[
    \sigma(A,B)\times [A\to (C_0\times C_1)] \times [B\to (C_0 \times C_1)] \to (C_0\times C_1)
    \]
    but $[A\to (C_0\times C_1)]$ is isomorphic to $[A\to {C_0}] \times [A\to {C_1}]$, and similarly for 
    $[B\to (C_0\times C_1)]$. This means that the required morphism can easily be constructed from the morphisms we are given 
    \[
    \sigma(A,B)\times [A\to {C_0}] \times [B\to {C_0}] \to C_0
    \]
    and $\sigma(A,B)\times [A\to {C_1}] \times [B\to {C_1}] \to C_1$.
    \item Similarly, we require a morphism 
    \[
    \sigma(A,B)\times [A\to [D\to C]] \times [B\to [D\to C]] \to [D\to C] 
    \]
    but this can be constructed from the exponential of the map 
    \[
    \sigma(A,B)\times [A\to C] \times [B\to C] \to C
    \]
    \item This is essentially as \ref{lem:case-prod}.
\end{enumerate}
\end{proof}

\begin{lemma}\label{lemma:alg-sound}
    Let 
    $\sigma (A,B) = \prod_p (A\to \sem{p}) \to (B\to \sem{p}) \to \sem{p}$, 
    then $\sigma$ supports disjunction elimination for any object which is the semantics of a formula: $\sem{\phi}$.
\end{lemma}

\begin{proof}
    The proof is by induction on the structure of $\phi$. The base case is $\phi=q$ where $q$ is an atom. The required morphism
    \[
    (\prod_p (A\to \sem{p}) \to (B\to \sem{p}) \to \sem{p}) \times 
    (A\to \sem{q}) \times (B\to\sem{q}) \to \sem{q}
    \] 
    is projection onto the  $q$ component followed by evaluation. The inductive cases for conjunction and implication follow immediately from Lemma \ref{lem:disjunction-support-closure} as the interpretations are given by product and exponentiation in the topos. The inductive case for disjunction also follows from the same lemma, since if $\phi=\psi\vee\chi$, then 
    $\sem{\phi} = \prod_p (\sem{\psi}\to \sem{p}) \to (\sem{\chi}\to \sem{p}) \to \sem{p}$.
\end{proof}

It is not the case that arbitrary objects of the topos support disjunction elimination. Specifically, the object $0$, which corresponds to the constant functor with value the empty set does not do so. If $A=B=C=0$, then $\sigma(A,B)=\prod_p (A\to \sem{p}) \to (B\to \sem{p})\to \sem{p}$ is isomorphic to $\prod_p \sem{p}$, while $[A\to C]$ and $[B\to C]$ are both isomorphic to $1$, but 
$\prod_p \sem{p}$ may be non-empty and hence there may be no morphism $\prod_p \sem{p}\to 0$. 

Soundness of the interpretation is now immediate: 

\begin{proposition}\label{prop:presheaf-sound}
  Given a proof in intuitionistic propositional logic of $\Gamma\vdash\phi$ then we can inductively define a morphism in the presheaf topos (a natural transformation between functors) $\sem{\Gamma}\to\sem{\phi}$. The induction is based on both the structure of the proof, and the syntactic structure of the formulae within it. 
\end{proposition}

This soundness result sits in contrast to the soundness of the traditional interpretation of intuitionistic proof theory in topoi. In that case, the result allows proofs including propositional variables, which are to be interpreted as arbitrary objects of the topos. In our case, we are restricted to proofs with formulae built from a collection of atoms, each of which has a defined semantics as a particular object of the topos. 

There is, however a way of escaping this restriction. We will do this in two stages. First we consider a semantics based on validity, and then one which is proof relevant. 

For the validity-based semantics we take a semantics derived from our current semantics, but we equate all the elements of each 
$\sem{\phi}w$.

\begin{definition}
We define the validity-based semantics as follows
    \[\vsem{\phi} w = \{* \mid \exists x. x\in \sem{\phi} w \}\]
\end{definition}

\begin{lemma} \mbox{ }
\begin{enumerate}
    \item $\vsem{\phi}$ is a subfunctor of the functor $1$, where $1w = \{*\}$.
    \item $\vsem{\phi\wedge\psi} = \vsem{\phi}\times \vsem{\psi}$
    \item $\vsem{\phi\supset\psi} = \vsem{\phi}\to \vsem{\psi}$
    \item $\vsem{\phi} w = \{*\}$ iff $w\forces\phi$ in the sense of Sandqvist \cite{Sandqvist2015IL}.
\end{enumerate}
\end{lemma}

Now consider the map on objects of the topos, sending $U$ to 
$\prod_p (U \to \vsem{p}) \to \vsem{p}$. 
If $U$ is a subfunctor of $1$, then $\prod_p (U \to \vsem{p}) \to \vsem{p}$ is itself isomorphic to a subfunctor of $1$.

\begin{definition}\label{def:nucleusK}
    If $U$ is a subfunctor of $1$, then let $KU$ be the subfunctor of $1$ isomorphic to $\prod_p (U \to \vsem{p}) \to \vsem{p}$. $K$ is an operator on subfunctors of $1$.
\end{definition}

\begin{lemma}\label{lemma:K-is-nucleus} For all subfunctors $U$ and $V$ of $1$:
    \begin{enumerate}
        \item $U$ is a subfunctor of $KU$.
        \item $KKU$ is a subfunctor of $KU$, and hence $K$ is idempotent ($K^2 = K$). 
        \item $K(U\times V) = KU\times KV$.
    \end{enumerate}
\end{lemma}

It follows that $K$ is a nucleus on the locale of subfunctors of 
$1$ ordered by inclusion, see Johnstone 
\cite{Johnstone:StoneSpaces}, and hence that its fixpoints define a 
sublocale.  

\begin{definition}
    A subfunctor $U$ of $1$ is said to be {\em closed} if $KU=U$. We write $\Omega_K$ for the set of closed subfunctors of $1$ ordered by pointwise inclusion. 
\end{definition}

\begin{lemma}\label{lemma:sublocale-interpretation} \mbox{}
\begin{enumerate}
    \item For any atomic proposition $p$, $\vsem{p}$ is closed. 
    \item $1$ is closed. 
    \item If $U$ and $V$ are closed, then so is $U\times V$.
    \item If $U$ is closed and $W$ is any subfunctor of $1$, then $W\to U$ is closed. 
    \item An arbitrary product of closed subfunctors of $1$ is closed.
    \item If $U$ and $V$ are closed, then the least closed subfunctor of $1$ containing both $U$ and $V$ is $\prod_p (U\to\vsem{p})\to (V\to\vsem{p}) \to\vsem{p}$, and this is therefore the join of $U$ and $V$ in $\Omega_K$.
    \item $\Omega_K$ is a complete Heyting algebra (and hence an algebra supporting the interpretation of Intuitionistic Propositional Logic). 
\end{enumerate}
\end{lemma}

\begin{proof}
    (2)-(5) hold for arbitrary kernels, and follow from the properties of \ref{lemma:K-is-nucleus}. (1) and (6) are specific to $K$, but follow from Lemma \ref{lemma:K-is-nucleus} and basic properties of types and intuitionistic logic. 
    
    Specifically, (1) follows from Lemma \ref{lemma:K-is-nucleus} (1) and the fact that if 
    $\forall p. (\vsem{q}\to\vsem{p}) \to\vsem{p}$, then in particular $(\vsem{q}\to\vsem{q}) \to\vsem{q}$, and since 
    $\vsem{q}\to\vsem{q}$ always holds, then $\vsem{q}$. This establishes that $K\vsem{q}$ is a subfunctor of $\vsem{q}$ as well as conversely, and hence is equal to it. 

    Similarly for (6) we suppose $U\leq W$, $V\leq W$ and that $W$ is closed. We show that 
    \[\prod_p (U\to\vsem{p})\to (V\to\vsem{p}) \to\vsem{p} \leq 
    \prod_p (W\to\vsem{p})\to\vsem{p} = W\]
    The proof is componentwise, for any $p$, we show that
    \[(U\to\vsem{p})\to (V\to\vsem{p}) \to\vsem{p} \leq (W\to\vsem{p})\to\vsem{p}\]
    This follows by monotonicity. Since $U\leq W$, $W\to\vsem{p}\leq U\to\vsem{p}$, and since $V\leq W$, 
    $W\to\vsem{p}\leq V\to\vsem{p}$. Hence 
    $(U\to\vsem{p})\to (V\to\vsem{p}) \to\vsem{p}$ is contained in 
    $(W\to\vsem{p})\to (W\to\vsem{p}) \to\vsem{p}$ which is equal to $(W\to\vsem{p}) \to\vsem{p}$. 
\end{proof}

$\Omega_K$ therefore provides a model of 
intuitionistic propositional logic, with the connectives interpreted as standard (meet, join, heyting implication, etc). Moreover there is an immediate 
connection with Sandqvist's interpretation of Intuitionistic 
Propositional Logic. The interpretations of the connectives are 
exactly as in Sandqvist's definitions, and for any formula $\phi$, 
we have $\vsem{\phi} w = \{*\}$ if and only if $w\forces\phi$. 
Sandqvist's interpretation of disjunction is the join operator in 
$\Omega_K$.

This can be further extended to a proof-sensitive account. The 
nucleus $K$ on subfunctors of $1$ internalizes to an  
endomorphism of the object $\Omega$ of truth values for the 
topos of presheaves:
\[
k = \lambda \omega. \forall p. (\omega\to\vsem{p})\to\vsem{p} :
\Omega \longrightarrow\Omega
\]
The properties that make $K$ a nucleus, then become the properties for $k$ 
to be a Joyal-Tierney topology (see \cite{MM1992}), and we can obtain a 
proof-relevant model valued in the topos of sheaves: 
$\mathop{\mbox{\it Sh}}_k (\Set^{\worlds^{\op}})$. 
The subobjects of $1$ in this sheaf model are precisely the closed 
subfunctors of $1$ for the nucleus $k$. 

The sheaves for a topology form a full subcategory of the presheaf topos, closed under products, exponentials, and limits generally. This gives the category of sheaves limits and exponentials. The equalizer of $k$ and the identity on $\Omega$ gives a subobject of $\Omega$,  which acts as the subobject classifier for the category of sheaves. The category of sheaves therefore has coproducts. These, however, are not the same as the coproducts in the category of presheaves. This structure can be used to give a proof-relevant model. 

However, critically, the interpretation of atomic propositions, $\sem{p}$, that we have given does not in general form sheaves, and hence the model that we have given is not actually a sheaf-theoretic one. 

To see why not, let's suppose we restrict to a setting in which we have two atomic propositions, $p$ and $q$. There are eight possible first-order rules for these propositions. Four of them make sense from a proof-theoretic perspective: 
\[
    \Rightarrow p,\quad \Rightarrow q,\quad q\Rightarrow p,\quad p\Rightarrow q
\]
The remaining four are valid in format, but do not enable us to derive any new information: 
\[
    p\Rightarrow p,\quad q\Rightarrow q,\quad p,q\Rightarrow p,\quad p,q\Rightarrow q
\]
A base is given by a subset of these, and there are therefore $2^8=256$ bases. 
The minimal bases in which both $p$ and $q$ are valid are: 
\[
\begin{array}{c}
\Rightarrow p,\quad \Rightarrow q\\
\Rightarrow p,\quad p\Rightarrow q\\
\Rightarrow q,\quad q\Rightarrow p
\end{array}
\]
The base  $\{\Rightarrow p,\quad \Rightarrow q\}$ generates one 
proof of $p$, the proof via $\Rightarrow p$. The base 
$\{\Rightarrow q,\quad q\Rightarrow p \}$ generates a different 
proof of $p$. Both of these bases are contained in 
$\{ \Rightarrow p, \quad\Rightarrow q,\quad q\Rightarrow p \}$, 
which contains both proofs. As a result, it is impossible to 
choose a single proof of $p$ that is valid for all bases that prove both $p$ and $q$.  
However, this is precisely what we would get if $\sem{p}$ were 
a sheaf in this example. 

Those familiar with program semantics will recognise the construct 
$((\ )\to \vsem{p})\to \vsem{p}$ as an instance of 
$((\ )\to R)\to R$, and so a form of continuation semantics. 
In \cite{Reynolds:HistoryContinuations}, John Reynolds quotes 
Christopher Wadsworth as describing continuations as 
`the meaning of the rest of the program'. Program semantics faces 
the issue that the meaning of a program is intended to be derived 
structurally, and so in terms of the meaning of its parts. However, it 
is not possible to run part of a program, only a complete one. As a 
result, to give meaning to part of a program, you have to say what 
that part will do in the context of its completion by the rest of a program; 
that is, in terms of what it will do given a continuation. 
In the $((\ )\to R)\to R$ 
construct, $R$ represents the end result of the program. 
$(\ )\to R$ represents the continuation, something that will 
convert the meaning of a fragment into a final result, and 
$((\ )\to R)\to R$ expresses that the meaning of a program fragment
is supposed to be something that takes a continuation and produces 
a final result. However, there is a hole here: if the meaning of a program 
fragment is to be $(X\to R)\to R$, then this suggests continuations 
should have type $((X\to R)\to R)\to R$, and program fragments should 
have another iteration, opening up an infinite loop; but, in intuitionistic type 
theory, there is a natural morphism 
\[
    (((X\to R)\to R)\to R) \longrightarrow (X\to R)
\]
allowing escape from the recursion. 

The analogy for us is direct. The meaning of the proof of a 
non-atomic formula $\phi$ is taken to be a mechanism that, 
given any atomic formula $p$ and the rest of a proof of $p$ from 
$\phi$, gives a proof of $p$. So, the complete proofs that are the 
analogues of programs are the proofs of atomic formulas, not proofs 
of arbitrary ones.

\section{Logical Metatheory} \label{sec:metatheory} 

We now have all that is required to complete our categorical
account of Sandqvist's proof-theoretic semantics for intuitionistic 
propositional logic. 

In Section~\ref{sec:categorical}, we established the (algebraic) soundness and 
completeness of our categorical set-up with respect to Sandqvist's 
formulation of validity. In this section, we give the two key meta-theoretic 
results, namely soundness and completeness with respect to the natural deduction 
system NJ. We defer a discussion of the choices around the interpretation of 
disjunction to Section~\ref{sec:disjunction}. 

\subsection{Soundness} \label{subsec:soundness} 

We establish the soundness of NJ \cite{Prawitz1965}, as given in Figure~\ref{fig:nj}, 
with respect to proof-theoretic validity. The proof proceeds by induction over the 
structure of proofs in NJ. While this result does not depend directly upon the soundness 
and completeness results for the algebraic interpretation, it nevertheless makes 
essential use of the characterization of validity in terms of natural transformations. 


\begin{theorem}[Soundness] \label{thm:CatSoundness}
If $\Gamma \vdash \phi$, then $\Gamma \models \phi$. 
\end{theorem} 
\begin{proof}
The proof, which is very similar to that of Lemma~\ref{lem:alg-sound}, 
is by induction over the structure of the derivation of 
$\Gamma \vdash \phi$. The critical cases are disjunction introduction 
and elimination.
\begin{description}[leftmargin=5mm,nosep]
\item[$\vee I$] The rules are
  \[
    \infer{\Gamma \vdash \phi \vee \psi}{\Gamma \vdash \phi}
        \quad \raisebox{2.5mm}{\mbox{\rm and}} \quad
    \infer{\Gamma \vdash \phi \vee \psi}{\Gamma \vdash \psi}
  \]
  We consider only the first rule; the proof for second rule is similar.
By  assumption, there is a natural transformation $\eta_\phi \colon
\sem{\Gamma} \rightarrow \sem{\phi}$. We define, using an informal $\lambda$-calculus 
notation, a natural
transformation $\eta \colon \sem{\Gamma} \rightarrow \sem{\phi \vee \psi}$ by
\[
\eta(\gamma) = \lambda p \,.\, \lambda F_1 \colon \sem{\phi} \supset
\sem{p} \,.\, \lambda F_2 \colon \sem{\psi} \supset \sem{p} \,.\, F_1(\eta_\phi \gamma)  
\]
  \item[$\vee E$] The rule is
    \[
      \infer{\Gamma \vdash \chi}
            {\Gamma \vdash \phi \vee \psi & \Gamma, \phi \vdash \chi &
              \Gamma, \psi \vdash \chi}
          \]
          By the induction hypothesis, there are 
   natural transformations $\eta_\phi\colon \sem{\Gamma, \phi}
   \rightarrow \sem\chi$,     $\eta_\psi \colon \sem{\Gamma, \psi}
   \rightarrow \sem\chi$ and
    $\eta_{\phi \vee \psi} \colon \sem{\Gamma} \rightarrow \sem{\phi \vee
      \psi}$. For any natural transformation $\eta \colon \sem{\Gamma}
    \times \sem{\chi} \rightarrow \sem{\sigma}$, we write
    $\mathop{Cur}(\eta) \colon  \sem{\Gamma}
    \rightarrow \sem{\chi} \supset \sem{\sigma}$ for the natural
    transformation obtained by applying the adjunction defining function
    spaces to $\eta$. 
    
    We use now an induction over $\chi$.
    \begin{description}[leftmargin=5mm,nosep]
    \item[$\chi$ = $p$, $p$ \mbox{\rm atom}]
      We define a natural transformation $\eta \colon \sem{\Gamma}
      \rightarrow \sem{\chi}$ by
      $\eta(\gamma) = \eta_{\phi \vee \psi}p
      (\mathop{Cur}(\eta_\phi)\gamma)(\mathop{Cur}(\eta_\psi)\gamma)$.
\item[$\chi = \chi_1 \wedge \chi_2$]
We also have $\Gamma, \phi \vdash \chi_i$ and $\Gamma \vdash \chi_i$
for $i = 1, 2$.
Hence, by the induction hypothesis, there are natural transformations
$\eta_i \colon \sem{\Gamma} \rightarrow \sem{\chi_i}$ for $i = 1, 2$. 
The natural transformation $\langle\eta_1, \eta_2\rangle$ is a natural
transformation from $\sem{\Gamma}$ to $\sem{\chi_1 \wedge \chi_2}$.
    \item[$\chi = \chi_1 \supset \chi_2$]
      We also have $\Gamma, \chi_1, \phi \vdash \chi_2$, $\Gamma,
      \chi_1, \psi \vdash \chi_2$ and $\Gamma, \chi_1 \vdash \phi \vee
      \psi$. By induction hypothesis there is a natural transformation
      $\eta \colon \sem{\Gamma} \times \sem{\chi_1} \rightarrow
      \sem{\chi_2}$. The natural transformation $\mathop{Cur}(\eta)$
      is a natural transformation from $\sem{\Gamma}$ to $\sem{\chi_1
        \supset \chi_2}$.
    \item[$\chi = \chi_1 \vee \chi_2$]
      For any atom $p$ and any $F_1 \colon \sem{\chi_1} \supset
      \sem{p}$, $F_2 \colon \sem{\chi_2} \supset
      \sem{p}$, we define a natural transformation $\mu^\phi \colon
      \sem{\Gamma} \times \sem{\phi} \rightarrow \sem{p}$ by
      $\mu^\phi(\gamma, t) = \eta_\phi(\gamma, t)pF_1F_2$. We
      similarly define a natural transformation $\mu^\psi \colon
      \sem{\Gamma} \times \sem{\psi} \rightarrow \sem{p}$ by
      $\mu^\psi(\gamma, t) = \eta_\psi(\gamma, t)pF_1F_2$.
      Now we define, again using an informal $\lambda$calculus notation, a natural transformation 
      $\eta \colon \sem{\Gamma} \rightarrow \sem{\chi_1 \vee \chi_2}$ by
      \[
      \begin{array}{rcl}
        \eta(\gamma)\! & \!=\! & \!\lambda p \,.\, \lambda F_1 \colon\sem{\chi_1} \supset
      \sem{p} \,.\, \lambda F_2 \colon \sem{\chi_2} \supset
      \sem{p} \,.\, \eta_{\phi \vee \psi} \gamma p (\mathop{Cur}(\mu^\phi)\gamma)(\mathop{Cur}(\mu^\psi)\gamma)
      \end{array}
      \]
    \end{description}
  \end{description}
\end{proof}

\subsection{Completeness} \label{subsec:completeness}

We now establish the completeness of proof-theoretic validity 
with respect to NJ. The proof begins with a construction that analogous to the construction of a 
term model, but one which is with respect to (proof-theoretic) 
validity. The completeness statement then follows by an appeal to naturality. 

For completeness, we seek to establish that if $\Gamma \Vdash \phi$, 
then $\Gamma \vdash \phi$. We make essential use of Sandqvist's flattening
and the associated base $\mathcal{N}_{-}$, as defined in Definition~\ref{def:N} of 
Section~\ref{sec:p-ts-IPL}. 




\begin{lemma} 
\label{lemma:logCompl}
Let $\phi$ be any formula. Let $\Delta^\flat$ be the flattening of $\phi$. Let $\mathcal{N}_{\Delta^\flat}$ 
be the special base given by Definition~\ref{def:N}.
 Let $\mathcal{B}$ be any base such that $\mathcal{B} \supseteq
  \mathcal{N}_{\Delta^\flat}$. 
For any set of atoms $Q$, let $\mathcal{W'} = \mathcal{W}_{\mathcal{B}, (Y \col\, Q)}$.
\begin{enumerate}
\item There exists a natural transformation from $\semwp{\phi}$ to
  $\semwp{\flatOp{\phi}}$.
  \item There exists a natural transformation from
    $\semwp{\flatOp{\phi}}$ to $\semwp{\phi}$.
\end{enumerate}
\end{lemma}
\begin{proof}
We prove both statements simultaneously by induction over the structure of $\phi$.
\begin{enumerate}
    \item
    \begin{description}[leftmargin=5mm,nosep]
    \item[$\phi = p$] By definition, $\flatOp{p} =p$, hence the identity function has the desired properties.
    \item[$\phi = \phi_1 \wedge \phi_2$] 
    Let $f$ be an element of $\semwp{\phi_1 \wedge \phi_2}(\mathcal{C}, (X\!:\!P))$. Hence $\pi_1(f)$ and $\pi_2(f)$ 
    are elements of $\semwp{\phi_1}(\mathcal{C}, (X\!:\!P))$ and $\semwp{\phi_2}(\mathcal{C}, (X\!:\!P))$, respectively, 
    where where $\pi_1$ and $\pi_2$ are the projections from $\semwp{\phi_1 \wedge \phi_2}$ to $\semwp{\phi_1}$ and $\semwp{\phi_2}$, respectively.
    By the induction hypothesis, there exist functions $\eta_1 : \semwp{\phi_1} \rightarrow \semwp{\flatOp{\phi_1}}$ and $\eta_2 : \semwp{\phi_2} \rightarrow \semwp{\flatOp{\phi_2}}$. By definition, $\eta_1(\pi_1(f))$ and $\eta_2(\pi_2(f))$ are derivations $P \vdash_{\mathcal{C}} \flatOp{\phi_1}$ and  $P \vdash_{\mathcal{C}} \flatOp{\phi_2}$, respectively. Using the $\wedge I$-rule of $\mathcal{N}_{\Delta^\flat}$, there is also a derivation $P \vdash_{\mathcal{C}} \flatOp{(\phi_1 \wedge \phi_2)}$.
    \item[$\phi = \phi_1 \supset \phi_2$] 
    Let $f$ be an element of $\semwp{\phi_1 \supset \phi_2}(\mathcal{C}, (X\!:\!P))$. 
    We start by showing that $P, \flatOp{\phi_1} \vdash_{\mathcal{C}} \flatOp{\phi_2}$. Suppose that $\vdash_{\mathcal{D}} P$ and $\vdash_{\mathcal{D}} \flatOp{\phi_1}$, for $\mathcal{D} \supseteq \mathcal{C}$.
    Then we also have $P \vdash_{\mathcal{D}} \flatOp{\phi_1}$.
    By the induction hypothesis, there exists a function $\eta_1$ from $\semwp{\flatOp{\phi_1}}(\mathcal{D}, X\!:\!P))$ to $\semwp{\phi_1}(\mathcal{D}, X\!:\!P))$. Hence there is also an element $g$ of $\semwp{\phi_1}(\mathcal{D}, X\!:\!P))$. 
    Therefore, $f(\iota, g)$, where $\iota$ is the reverse inclusion of $\mathcal{C}$ into $\mathcal{D}$, is an element of $\semwp{\phi_2}(\mathcal{D}, (X\!:\!P))$.
    By the induction hypothesis, there is also a function $\eta_2$ from $\semwp{\phi_2}(\mathcal{D}, X\!:\!P))$ to $\semwp{\flatOp{\phi_2}}(\mathcal{D}, X\!:\!P))$. 
    Hence there is a derivation $P \vdash_{\mathcal{D}} \flatOp{\phi_2}$, and therefore by cut, also a derivation $\vdash_{\mathcal{D}} \flatOp{\phi_2}$. The $\supset I$-rule of $\mathcal{N}_{\Delta^\flat}$ now produces a derivation $P \vdash_{\mathcal{C}} \flatOp{(\phi_1 \supset \phi_2)}$.
    \item[$\phi = \phi_1 \vee \phi_2$]
   Let $f$ be an element of $\semwp{\phi_1 \vee \phi_2}(\mathcal{C}, (X\!:\!P))$. By definition, there exists also an element of $\semwp{(\phi_1 \supset \flatOp{(\phi_1 \vee \phi_2)}) \supset (\phi_2 \supset \flatOp{(\phi_1 \vee \phi_2)}) \supset  \flatOp{(\phi_1 \vee \phi_2})}(\mathcal{C}, (X\!:\!P))$. By the induction hypothesis, there exists also a derivation 
   $P \vdash_{\mathcal{C}} (\flatOp{\phi_1} \supset \flatOp{(\phi_1 \vee \phi_2)}) \supset (\flatOp{\phi_2} \supset \flatOp{(\phi_1 \vee \phi_2)}) \supset  \flatOp{(\phi_1 \vee \phi_2)}$. Now an application of the $\vee I$-rules of $\mathcal{N}_{\Delta^\flat}$ yields $P \vdash_{\mathcal{C}} \flatOp{(\phi_1 \vee \phi_2)}$. 
    \end{description}
    \item 
    \begin{description}
    \item[$\phi = p$] 
    By definition, $\flatOp{p} =p$, hence the identity function has the desired properties.
    \item[$\phi = \phi_1 \wedge \phi_2$] Suppose $P \vdash_{\mathcal{C}} \flatOp{(\phi_1 \wedge \phi_2)}$. Using the $\wedge E$-rule of $\mathcal{N}_{\Delta^\flat}$ we obtain derivations $\Phi_1$ of $P \vdash_{\mathcal{C}} \flatOp{\phi_1}$ and $\Phi_2$ of 
    $P \vdash_{\mathcal{C}} \flatOp{\phi_2}$ respectively. 
        By the induction hypothesis, there exist functions $\eta_1 : \semwp{\flatOp{\phi_1}}(\mathcal{C}, (X\!:\!P)) \rightarrow \semwp{\phi_1}(\mathcal{C}, (X\!:\!P))$ and $\eta_2 : \semwp{\flatOp{\phi_2}}(\mathcal{C}, (X\!:\!P)) \rightarrow \semwp{\phi_2}(\mathcal{C}, (X\!:\!P))$. Hence $(\eta_1(\Phi_1), \Phi_2)$ is an element of $\semwp{\phi_1 \wedge\phi_2}(\mathcal{C}, (X\!:\!P))$.
 \item[$\phi = \phi_1 \supset \phi_2$]
 By the definition of function spaces, it suffices to show the existence of a function from $\semwp{\flatOp{(\phi_1 \supset \phi_2)}}(\mathcal{C}, (X\!:\!P)), \semwp{\phi_1}(\mathcal{C}, (X\!:\!P))$ to $\semwp{\phi_2}(\mathcal{C}, (X\!:\!P))$.
 Let $\Phi$ be a derivation $P \vdash_{\mathcal{C}} \flatOp{(\phi_1 \supset \phi_2)} $ and $f$ be an element of $\semwp{\phi_1}(\mathcal{C}, (X\!:\!P))$. By the induction hypothesis, there exists a function $\eta_1:  \semwp{\phi_1}(\mathcal{C}, (X\!:\!P)) \rightarrow \semwp{\flatOp{\phi_1}}(\mathcal{C}, (X\!:\!P))$, and therefore also a derivation $P \vdash_{\mathcal{C}} \flatOp{\phi_1}$. The $\supset E$-rule of $\mathcal{N}_{\Delta^\flat}$ now yields a derivation of $P \vdash_{\mathcal{C}} \flatOp{\phi_2}$. By the induction hypothesis, there exists a function $\eta_2: \semwp{\flatOp{\phi_2}}(\mathcal{C}, (X\!:\!P)) \rightarrow \semwp{\phi_2}(\mathcal{C}, (X\!:\!P))$, and therefore also an element  of $\semwp{\phi_2}(\mathcal{C}, X\!:\!P))$.
 \item[$\phi = \phi_1 \vee \phi_2$]
  It suffices to show that, for all atoms $p$, there is a function from $\semwp{\flatOp{(\phi_1 \vee \phi_2)}}(\mathcal{C}, (X\!:\!P)), \semwp{\phi_1 \supset p}(\mathcal{C}, (X\!:\!P)),  \semwp{\phi_2 \supset p}(\mathcal{C}, (X\!:\!P))$ to $\semwp{p}(\mathcal{C}, (X\!:\!P))$. 
    Suppose $P \vdash_{\mathcal{C}} \flatOp{(\phi_1 \vee \phi_2)}$ and let $f_1$ and $f_2$ be elements of $\semwp{\phi_1 \supset p}(\mathcal{C}, (X\!:\!P))$ and $ \semwp{\phi_2 \supset p}(\mathcal{C}, (X\!:\!P))$ respectively. By the induction hypothesis, there exists also derivations $P \vdash_{\mathcal{C}} \flatOp{(\phi_1 \supset p)}$ and   $P \vdash_{\mathcal{C}} \flatOp{(\phi_2 \supset p)}$. The $\vee E$-rule of $\mathcal{N}_{\Delta^\flat}$ now yields a derivation of $P \vdash_{\mathcal{C}} p $. 
 \end{description}
\end{enumerate}
This completes the proof. 
\end{proof}

\begin{theorem}[Completeness] \label{thm:CatCompleteness}
If $\Gamma \models \phi$, then $\Gamma \vdash \phi$. 
\end{theorem} 
\begin{proof}
 Let $\Delta^\flat$ be the flattening of $\Gamma$ and $\phi$. 
Let $\mathcal{N}_{\Delta^\flat}$ be the special base given by Definition~\ref{def:N} for this set of atoms.
By Lemma~\ref{lemma:logCompl}, there exists a function $\eta$ from $\sem{\flatOp{\Gamma}}(\mathcal{N}_{\Delta\flat}, (X:\flatOp{\Gamma}))$ to $\sem{\flatOp{\phi}}(\mathcal{N}_{\Delta^\flat}, (X:\flatOp{\Gamma}))$. 
By definition, $\eta(id)$ is a derivation $\flatOp{\Gamma} \vdash_{\mathcal{N}_{\Delta^\flat}} \flatOp{\phi}$. 
Theorem~5.1 of Sandqvist now implies $\Gamma \vdash \phi$ in NJ.
\end{proof}



\section{Disjunction and its Interpretation} 
\label{sec:disjunction} 

There is only one interpretation of implication in any
given semantics. This is not the case for disjunction. Often
there are several ways of interpreting disjunction for a particular
semantics with sometimes significantly different properties. In this
section, we discuss possible alternatives for interpreting disjunction
and give a justification for the interpretation of disjunction used
for base-extension semantics.

The validity of formulae is commonly defined by an induction over the structure 
of formulae. The clause for disjunction usually states --- for example, in 
elementary Kripke-style semantics --- that a disjunction $\phi \vee \psi$ is 
valid iff $\phi$ or $\psi$ are valid. This clause is problematic for a notion 
of validity using a proof-theoretic semantics, as completeness does not hold, 
as shown in \cite{Piecha2015failure}: specifically, the authors show that with 
this definition of validity, Harrop's rule is valid; however, Harrop's rule 
is not derivable in intuitionistic logic.  

From the inferentialist perspective, Kripke's clause is too strong because it assumes that the suasive content of a disjunction is identical to that of its disjuncts. However, Sandqvist's treatment, corresponding to the $\vee$-elimination rule of NJ for disjunction, expresses that whatever can be inferred from both disjuncts can be inferred from the disjunction. It is this correspondence that allows completeness to go through. 

Categorically, disjunction is typically interpreted as a coproduct. But 
coproducts are defined as a left adjoint. From the logical perspective, 
this means they are determined by the elimination rule, and hence that a 
characterization in terms of the introduction rule does not fit with the 
categorical interpretation, absent some strong properties on the 
indecomposability of formulae.

Instead, then, the validity of a disjunction is given by the elimination rule: if 
$\phi \vee \psi$ is valid, and for all formulae $\chi$, if $\phi$ is valid 
implies $\chi $ is valid and if $\psi$ is valid implies $\chi$ is valid, 
then $\chi$ is valid. Sandqvist's definition of validity for disjunction 
restricts $\chi$ to atoms (cf. \cite{Tennant78entailment,Tennant2017core}). In this way, we obtain an 
inductive definition of validity. Our categorical interpretation models 
this definition of validity. Completeness can be obtained by using a basis 
$\mathcal{N}$, which has a separate atom denotating each formula. It follows 
that the quantification over atoms captures quantification over all formulae. 

We can also use the categorical definition of validity --- as set 
up in Definition~\ref{def:cat-base-val} --- to show that the interpretation of disjunction 
by coproducts leads to incompleteness by proving that completeness would imply that the rule 
    \[
        \infer{(p \supset q) \vee (p \supset r)}{p \supset q \vee r}      
    \]
was derivable. This rule is a strong disjunction property and it is not
derivable in IPC. The argument in the proof of the theorem is a
reformulation of an argument by Sandqvist~\cite{Sandqvist2015IL}: 
\begin{theorem} \label{thm:interpretation-of-v} 
      Suppose in Definition \ref{def:int-func}, the definitional clause for $\sem{\phi\vee\psi}$ is replaced by 
      \begin{itemize}
          \item[--] $\sem{\phi \vee \psi}$ is the coproduct  $\sem{\phi} + \sem{\psi}$.
      \end{itemize} 
      Then there is a natural
      transformation from $\sem{p \supset (q \vee r)}$ to $\sem{(p
      \supset q) \vee (p \supset r)}$.
\end{theorem}

\begin{proof}
  Consider the rule 
   \[
        \infer{(p \supset q) \vee (p \supset r)}{p \supset q \vee r}      
   \]
  
  Now consider any $(\mathcal{B}, (X\!:\!P))$ that satisfies $p \supset (q \vee r)$, 
  which means $\sem{p \supset (q \vee r)}(\mathcal{B}, (X \col P))$ is 
  non-empty.  Hence there is a natural transformation $\eta$ from 
  $\mathcal{W}(-, (\mathcal{B}, (X\!:\!P)))\times \sem{p}$ to $\sem{q \vee r}$.
  Let $\mathcal{C}$ be the base that is obtained by adding the rule
  $\Rightarrow\! p$ to $\mathcal{B}$. Let $\Phi$ the derivation of $P
  \vdash_{\mathcal{C}} p$ in $\mathcal{C}$ obtained by applying this rule. 
  Let $\iota$ be the reverse inclusion map from $(\mathcal{C}, P)$ to $(\mathcal{B}, P)$.
  Hence, $\eta_{(\mathcal{C}, P)}(\iota, \Phi)$ is an element of $\sem{q \vee
    r}(\mathcal{C}, (X\!:\!P)) = \sem{q}(\mathcal{C}, (X\!:\!P)) + \sem{r}(\mathcal{C}, (X\!:\!P))$.
  Therefore, $\eta_{(\mathcal{C}, P)}(\iota, \Phi)$ is an element of
  $\sem{q}(\mathcal{C}, (X\!:\!P))$ or $\sem{r}(\mathcal{C}, (X\!:\!P))$. 
  
  Now consider the first case of this disjunction; that is, in which there is a derivation of 
  $X\!:\!P \vdash_{\mathcal{C}} q$. We must show that there is an element of 
  $\sem{p \supset q}(\mathcal{B}, (X\!:\!P)) = Nat(\mathcal{W}(-,
  (\mathcal{B}, (X\!:\!P)) \times \sem{p}, \sem{q})$.
  Consider any base $\mathcal{D} \supseteq \mathcal{B}$ and set of
  atoms $Q$ such that $\Phi \col Y:Q \vdash_{\mathcal{D}} P$ and there
  exists a derivation $\Psi \col Y:Q \vdash_{\mathcal{D}} p$. 

  Now we show by induction over the derivation of $Z:S, X\!:\!P \vdash_{\mathcal{C}} s$ 
  that there exists a derivation $Z:S, Y:Q \vdash_{\mathcal{D}} s$. 
  
  For any rule other than $\Rightarrow p$, the substitution of $\Phi$ for $X$ yields a 
  derivation $Z:S, Y:Q \vdash_{\mathcal{D}} s$. Now suppose the last rule is 
  $Z:S, X\!:\!P \Rightarrow p$. By assumption, there is also a derivation of
  $Z:S, Y:Q \vdash_{\mathcal{D}} p$.
  
  By a similar argument, we show that in the second case there is an
  element of $\sem{p \supset r}(\mathcal{B}, (X\!:\!P))$. It therefore follows that 
  there is also an element of 
  \[
      (\sem{p \supset q} \vee \sem{p \supset r})(\mathcal{B}, (X\!:\!P) 
  \] 

\vspace{-5.5mm}  
\end{proof}

It follows that completeness cannot hold if disjunction is interpreted by
coproducts (cf. the distinction between coproducts in presheaves and sheaves, 
as discussed in Section~\ref{sec:kripke}). This proof can be generalized by 
replacing the atom $p$ by any formula $\phi = p_n \supset \cdots \supset p_1 \supset p_0$, 
where all $p_i$s are atoms. The modification of the proof consists of replacing $\sem{p}$ 
by $\sem{ p_n \supset \cdots \supset p_1 \supset p_0}$ and using this natural 
transformation instead of the rule $\Rightarrow p$ to construct a derivation 
of $Z:S, Y:Q \vdash_{\mathcal{D}} s$.

The proof exploits the fact that the extensions in a base extension
semantics have more structure. In particular, for each base 
$\mathcal{B}$ and for each formula $\phi$ as specified above one can construct a rule
that corresponds to adding a derivation of $\phi$. Furthermore, 
there is a particular extension $\mathcal{C}$ of the base $\mathcal{B}$ 
(namely the one that just adds this rule) such that whenever an atom is 
derivable in $\mathcal{C}$, it is derivable in any extension of $\mathcal{B}$, 
which satisfies the formula $\phi$. If disjunction is modelled by coproducts, 
this implies a strong disjunction property: that there is a morphism from
$\sem{\phi \supset (q \vee r)}(\mathcal{B}, P)$ to $\sem{(\phi \supset
  q) \vee (\phi \supset r)}(\mathcal{B}, P)$.

This particular extension $\mathcal{C}$ does not necessarily exist
in a Kripke semantics, hence this argument does not hold for Kripke 
semantics.









\subsection*{Acknowledgments} 

We are grateful to the anonymous reviewers and to Alexander Gheorghiu and Tao Gu for helpful comments on this work.

\bibliographystyle{plain}
\bibliographystyle{plain}

\appendix 

\section{Remarks on Consequence Relations}
\label{sec:consequence-relations}

It will turn out that the significance of the base $\mathcal{N_{-}}$ goes way beyond its use
in Sandqvist's completeness theorem. Specifically, the algebraic interpretation 
that we give in Section~\ref{sec:categorical} makes essential use of the canonical proofs 
of atomic propositions of the form $\phi^\flat$ available in $\mathcal{N_{-}}$ in order to 
define the natural transformations that we needed in our interpretation.  


As we noted, one of the key issues for proof-theoretic 
semantics is its completeness, or lack of it \cite{Piecha2016completeness,Piecha2015failure}. 
However, this issue is more delicate than it at first seems. 

To begin with, we need to be clear about the judgements being used. Different 
forms of judgement give different answers to questions of soundness and 
completeness. At the most basic level, we have validity: the basic judgement 
is whether a formula is valid. Soundness means that theorems (provable 
formulae) should be valid in the semantics and completeness that any 
formula which is always valid in the semantics should be a theorem. 
This, however, is a fairly weak notion, and the judgements used are often 
strengthened to sequents. In this case, the soundness and completeness 
correspond to preservation and reflection of the validity of logical 
consequences, and there is more need to be specific how things 
are set up. This occurs even when we are just considering which consequences are valid. 
There is a further level of complexity when we come to consider possible 
representations of validations. 

The common standard formulation is that a consequence relation is a 
relation expressing that a formula is a consequence of a set of other 
formulae. This concept is abstractly captured in the form of a reflexivity 
property, a monotonicity property, and a compositionality property,  
the last corresponding to cut. 

\begin{definition}
A (finitary) consequence relation is a relation \(\Gamma\cons p\) 
between a (finite) set of formulae and a formula, such that: 
\begin{enumerate}
    \item[--] $\phi\cons \phi$
    \item[--] if $\Gamma \cons \phi$, then $\Gamma,\psi \cons \phi$
    \item[--] if $\Gamma \cons \phi$ and $\Delta,\phi \cons \psi$, then 
    $\Gamma,\Delta \cons \psi$.
\end{enumerate}
\end{definition}

There is a standard way of getting a consequence relation from a definition 
of validity. We write $\valid \phi$ to mean that $\phi$ is valid. 

\begin{definition}\label{def:cons-from-valid}
The \emph{consequence relation generated from a validity} is defined by 
$\Gamma \cons \psi$ iff  if $\valid\gamma$, for all $\gamma\in \Gamma$, then $\valid\psi$. 
\end{definition}

Note that we can recover validity from consequence:
\[\mbox{$\valid \psi$ iff $\emptyset\cons \psi$}\]

However, different consequence relations can give rise to the same 
validity. We therefore need to be careful about which consequence relations 
we are using and why. But the consequence relation defined above is canonical 
in the following sense: 
\begin{lemma}\label{lem:finitary-consequence-max}
The finitary consequence relation generated from a validity is the largest 
finitary consequence relation that corresponds to that validity. 
\end{lemma}

This property does not hold if we allow consequences $\Gamma\cons\psi$ 
where $\Gamma$ is infinite (the largest consequence relation corresponding 
to a validity is $\Gamma\cons\psi$ iff $\Gamma$ is finite and $\psi$ is  a 
consequence of $\Gamma$ as per the finite case, or $\Gamma$ is infinite 
and $\psi$ is arbitrary). Moreover, the structures do not work well for 
substructural logic, and must be changed to a relation between an element 
of an algebra of formulae and a formula. 

Definition \ref{def:cons-from-valid} links strongly to the notion of an 
admissible rule. A rule is said to be admissible if its application 
preserves validity: if, in an instance, all hypotheses are valid, then so is 
the conclusion. But proof systems generate a different natural notion of 
consequence: the relation that expresses that a formula can be proved 
from a set of hypotheses. 


\end{document}